\documentclass[11pt,oneside,letterpaper]{article}
\usepackage{amsmath}
\usepackage{amsfonts}
\usepackage{amssymb}
\usepackage{amsthm}
\usepackage{graphicx}
\usepackage{multicol}
\usepackage{multirow}
\usepackage{booktabs}
\usepackage{color}
\usepackage{rotfloat}
\usepackage[T1]{fontenc}
\usepackage{amssymb,amsmath}
\usepackage{xr}
\usepackage[left=1in, bottom=1in, right=1in, top=1in]{geometry}
\usepackage{times}
\usepackage[numbers]{natbib}
\usepackage{cases}
\usepackage[latin1]{inputenc}
\usepackage{float}
\usepackage{pdfsync}
\usepackage{setspace}
\usepackage{enumerate}
\usepackage{rotating}
\numberwithin{equation}{section}
\newtheorem{definition}{Definition}
\newtheorem{lemma}{Lemma}[section]
\newtheorem{proposition}{Proposition}
\newtheorem{theorem}{Theorem}
\newtheorem{assumption}{Assumption}

\theoremstyle{definition}\newtheorem{example}{Example}
\theoremstyle{definition}\newtheorem{remark}{Remark}[section]

\DeclareMathOperator*{\arginf}{arg\,inf}
\DeclareMathOperator*{\argmax}{arg\,max}

\title{The Information Projection in Moment Inequality Models: Existence, Dual Representation, and Approximation}
\author{Rami V. Tabri \\
 Department of Econometrics and Business Statistics, Monash University, \\ 29 Ancora Imparo Way, Clayton, Victoria, 3800, Australia,\\ Email: rami.tabri@monash.edu.}
\begin{document}
\maketitle
\begin{abstract}
Relative entropy minimization is a widely used method in decisions and operations research that incorporates information through constraints on the underlying probability model. The solution is called information projection, and we 
present new results for its existence, exponential family representation, and approximation in the infinite-dimensional setting for moment inequality constraint sets, nesting both conditional and unconditional moments and allowing for an infinite number 
of such inequalities. Our approach considers the Fenchel dual of the relative entropy minimization problem, with a key innovation being the exhibition of the dual variable as a weak vector-valued integral, enabling the formulation of a simple 
approximation scheme. Under suitable assumptions, the values of finite-dimensional convex stochastic programs can approximate the dual problem's optimum value and that, in addition, every accumulation point of a sequence of optimal solutions for the approximating programs is an optimal solution of the dual problem. We illustrate the verification of assumptions and construction of the approximation scheme's parameters for the cases of unconditional and conditional first-order stochastic dominance constraints and conditions that characterize selectionable distributions for a random set, and present numerical experiments demonstrating the simplicity of the approximation scheme.
 \end{abstract}
OR/MS Subject Classification: (i) Probability: Relative entropy and distribution comparison; (ii) Programming: Constrained infinite-dimensional optimization and approximation. \\
Area of Review: Stochastic Models.

\section{Introduction}
\par Given a reference probability distribution (PD) $Q$, a frequently encountered question that arises in operations research and management science is to calculate a probability density function on the basis of some known moment restrictions that is ``closest'' to $Q$. One prevalent method for selecting an estimate from the set of all densities satisfying the prescribed moment constraints is to choose it to minimize the relative entropy (\citet{Kullback-Leibler}). When it exists, the solution of this minimization problem is known as the $I$-projection. This approach has been successfully applied to decision problems, such as portfolio selection (e.g.,~\citet{LASSANCE2023302,Post2017}), utility maximization (\citet{Frittelli2000}), and to operations research, such as process monitoring (e.g.~\citet{ALWAN1998526,BAJGIRAN2021196,CHANG2022108150}), marketing analytics (e.g.,~\citet{Brockett1995,Soo2002}), and visual analytics~(e.g.,~\citet{Chen2008}). It has also been widely applied in other areas such as machine learning (e.g.,~\citet{Dhilion,NIPS2015_06138bc5,COTFNT}), optimal transport (e.g.,~\citet{LEONARD,NutzWiesel2022}), statistics (e.g.,~\citet{Kullback,Sheehy-EL,haberman1984,Kortanek1993}), probability theory (e.g., \citet{Sanov,csiszar1975,Csiszar-conditional}), and 
other scientific fields.

\par Phrased mathematically, for a given reference PD $Q$ defined on the measurable space $(\Omega,\mathcal{F}),$ the relative entropy minimization problem this paper considers is
\begin{equation}\label{eq - KL Min Problem}
\begin{aligned}
\text{minimize}& \quad m(p)=\begin{cases} \int_{\Omega}p\log(p)\,dQ &\mbox{if } p\geq0,\quad \int_{\Omega}p\,dQ=1\\
\infty &\mbox{elsewhere},\end{cases} \\
\text{subject to}&\quad \int_{\Omega}f_\gamma\, p\,dQ\leq 0\;\forall\gamma\in\Gamma,\quad p\in L_1(Q),
\end{aligned}
\end{equation}
with the understanding that the integrals must be well-defined -- see Section~\ref{Section - Setup}. The choice variable is the density $p$ and $\{f_\gamma:\gamma\in\Gamma\}$ is a given set of moment functions with index set $\Gamma$ having infinite cardinality. The inequality restrictions we consider also include conditional moment inequality restrictions, as such restrictions are equivalently characterized in terms of unconditional ones using a continuum or countably infinite number of instrument functions --- a result put forward by~\citet{Andrews-Shi}. 

\par Many applications call for imposing multiple constraints that take the form of infinitely many moment inequality restrictions with known moment functions. Salient examples in decisions include shape restrictions such as (i) unconditional and conditional stochastic dominance orderings, which are fundamental to the rankings of strategies in many areas of practical significance, such as investment (\citet{Post2017}), water-conserving irrigation in agricultural processes (\citet{Harris-Mapp}), disease control in epidemiology (\citet{VERTERAMOCHIU2020104906}), and social policy evaluation (\citet{Foster-Shorrocks}), and (ii) martingale-type restrictions in stochastic finance (\citet{Frittelli2000}). Shape restrictions also arise in Bayesian inference problems, where one imposes such constraints on the posterior distribution of latent variables~(e.g.,~\citet{Ganchev2010}). This approach is particularly useful in models where incorporating a-priori information about latent variables is complex and intractable, such as graphical models. $I$-projections also arise in a Bayes factor approach to statistical testing problems involving $e$-variables and a composite alternative hypothesis (e.g.,~\citet{Grunwald2024}, Section~3). An example of the constraints in (\ref{eq - KL Min Problem}) relevant to that testing problem is stochastic dominance, which is specified under the alternative hypothesis to represent a distributional effect, against a null hypothesis of no effect or homogeneity.

\par In operations research, a practical problem is the monitoring and control of industrial processes with censored data (e.g.,~\citet{Steiner-MacKay2001}). In this setting one observes random sets (i.e., a set-valued random variable) rather than the response values because of the censoring. A fundamental result due to~\citet{Artstein1983DistributionsOR} allows us to characterize the distribution of bundles of random vectors that constitute the random set in terms of infinitely many moment inequality restrictions. Another practical problem in the same vein as the censored data example is when data on response values are missing as opposed to being censored (e.g.,~\citet{Du-Hu-Zhang}). In this setting, the response values' distribution is partially identified with identified set characterized by infinitely many moment inequality restrictions (\citet{manski2005}). Consequently, one could use relative entropy with these restrictions to construct control charts based on information theory for monitoring distributions of a process variable and process attributes with incomplete data along the lines of \citet{ALWAN1998526}, for instance. This approach would circumvent the use of potentially misspecified parametric models for incomplete data in practice, and hence, improve accuracy and reliability of monitoring and detection processes. To the best of our knowledge, this approach has not been considered in the operations research literature. In consequence, solving the relative entropy minimization problem this paper considers could significantly impact data applications in operations research and decisions.

\par In this paper, we develop the existence, exponential family representation, and approximation of the  $I$-projection problem~(\ref{eq - KL Min Problem}). We propose a combination of conditions on the reference PD $Q$ and the moment functions that facilitates these developments, which are straightforward to verify in practice. The conditions establish the existence and exponential family representation of the $I$-projection using results on Fenchel duality put forward by~\citet{Bhatacharya-Dykstra}. Furthermore, the dual problem's solution belongs to the $L_1(Q)$-norm closure of the positive linear span of the moment functions for which the moment inequalities are binding, where $L_1(Q)$ is the Lebesgue space of absolutely integrable functions with respect to $Q$. This expression for the $I$-projection is new though not surprising given the previously obtained results for the case with a finite number of inequalities. While there is a rich literature that has investigated the existence and dual representation of $I$-projections in the case of moment restrictions (e.g.,~\citet{csiszar1975,Borwein-Lewis-1991,Borwein-Lewis-1993}), the results of~\citet{Bhatacharya-Dykstra} are at the forefront, because they apply to such problems when the standard Fenchel dual approach fails.

\par Infinite-dimensional problems such as~(\ref{eq - KL Min Problem}) are difficult (if perhaps not impossible) to solve in closed-form. Consequently, one is forced to use some kind of approximation scheme in order to obtain the $I$-projection. We introduce an approximation scheme that approximates the $I$-projection in terms of the Fenchel dual of~(\ref{eq - KL Min Problem}). The conditions we propose enables the exhibition of the dual variable as a Gelfand-Pettis integral (\citet{Gelfand} and ~\citet{Pettis}): a weak vector-valued integral with respect to a positive Radon measure defined on the $L_1(Q)$-closed convex hull of the set of moment functions. This characterization of the dual variable is key in our setup because the weak integral can be approximated by ``Riemann sums'' in the $L_1(Q)$-norm, which forms the basis of the proposed approximation scheme. The weak integral representation of the dual variable introduces a reparametrization of the $I$-projection's dual problem in terms of these Radon measures, which we equip with their weak-star topology. The use of this topology is advantageous because we show the objective function in the reparametrized dual problem is weak-star lower semi-continuous and coercive, yielding the existence of a solution, which may not be unique. However, under our conditions, we prove that
\begin{enumerate}[(i)]
\item the dual problem's optimum value can be approximated by a sequence of optimal values of finite-dimensional convex programs, and
\item every weak-star accumulation point of a sequence of optimal solutions of these finite-dimensional convex programs is an optimal solution for the dual problem.
\end{enumerate}
Furthermore, the proposed approximation scheme is straightforward to implement numerically using off-the-shelf simulation-based methods, such as the Sample Average Approximation (e.g., see~\citet{book-SAA} and the references therein), which is an appealing feature for practice.

\par The problem of finding conditions under which an approximation scheme has properties (i) and (ii) is of great interest for practice. For infinite programs, in general, the notion of Gamma-convergence of sequences of functionals has been successfully used to develop such conditions, for example, as in~\citet{Dal_Maso} and~\citet{Mena-Lerma-2005}. The approximation of the Fenchel dual problem of~(\ref{eq - KL Min Problem}) can also be studied using the general results of~\citet{Dal_Maso} and~\citet{Mena-Lerma-2005}. However, our approach is more direct and natural given the structure of the approximation scheme. Furthermore, we establish (i) under strictly weaker conditions than those considered in~\citet{Dal_Maso} and~\citet{Mena-Lerma-2005}; see Remark~\ref{Remark on Gamma Convergence}.

\par The proposed approximation scheme is new and contrasts with leading approaches for $I$-projections. Such approximation schemes focus on scenarios where the constraint set can be expressed as a finite intersection of sets that are variation-closed and convex, and are based on cyclic iterated $I$-projections onto those sets (e.g.,~\citet{csiszar1975,Dykstra-AoP,Dykstra-Wollen-87,Bhattacharya1997}). In the infinite-dimensional setting,~\citet{bhattacharya2006} has put forward such an algorithm, with its implementation relying on the use of closed-form expressions for the $I$-projections onto the individual intersecting sets at each iteration. The constraint sets we consider have this finite intersection representation, but our general setup may not yield closed-form expressions for the $I$-projections onto the individual sets. Consequently, it may be infeasible to implement Bhattacharya's algorithm without imposing further restrictions on our $I$-projection problem. Another motivation for our approximation scheme comes from this point, where the additional restrictions for executing Bhattacharya's algorithm in our framework may be severe in practice, and hence, would reduce the scope of applications. We elaborate on this point in the paper using examples of marginal stochastic order considered by~\citet{bhattacharya2006}, who obtains closed-form expressions for the $I$-projection by imposing square-integrability restrictions on the $I$-projection problem. Contrastingly, our approximation scheme does not require this additional structure to be feasible in practice; therefore, it is more widely applicable for moment inequality models.

\par This paper is organized as follows. Section~\ref{Section - OR Example} illustrates an application of the $I$-projection problem~(\ref{eq - KL Min Problem}) using an example of process monitoring. Section~\ref{Section - Setup} introduces the setup for deriving the main results. Section~\ref{Section- Main Results} presents the main results and their proofs, and Section~\ref{Section - Discussion} presents a discussion of the scope of our main results and implications for practice. All technical lemmas are relegated to the Appendix for ease of exposition. Section~\ref{Section - Examples} illustrates the theory with numerical experiments for moment functions that define \citet{Artstein1983DistributionsOR} inequalities, characterizing selectionable distributions for a random interval, which is helpful for modeling censored data. In the Supplementary Material, we present technical results for the Examples in Section~\ref{Section - Discussion}, numerical experiments for moment functions that define unconditional and conditional first-order stochastic dominance constraints, and two counterexamples where assumptions are not verified and the main results do not hold.

\section{Illustrative Example: Process Monitoring with Interval Censored Data}\label{Section - OR Example}
Process monitoring methods have been and continue to be widely implemented in manufacturing, service, and scientific activities. A notable method in this literature is the one put forward by~\citet{ALWAN1998526}, who proposed a general framework for constructing control charts based on information theory for monitoring moments and distributions of a process variable. However, it falls short in applications with 
censored data. The optimization problem we tackle helps extend their approach to such scenarios without relying on an \emph{assumed} model for the incomplete data, circumventing the errors due to their incorrect misspecification. This section describes their information-theoretic process control (ITPC) method, and how to extend it to the setup with interval-censored data. 

\par In practice, the true distribution of a process variable is typically unknown. However, there is usually some partial information about it in the form of moments. For example, we may have the in-control mean and variance for a measurement variable by engineering design. In mathematical terms, let $P_0$ denote the true distribution of the process variable for the in-control state, defined on the measurable space $(\Omega,\mathcal{F})$, then the practitioner observes the moments $a_{k}=\int_{\Omega} g_{k}(\omega)\,dP_0(\omega),\quad k=0,1,\ldots,K,$ where $g_0(\omega)=1$ for all $\omega\in\Omega$ is used for normalizing the density, and $g_k$'s are absolutely integrable functions with respect to $P_0$. The process is then monitored using these initial moments and the data moments $d_{1},\ldots,d_K$ computed from sampling the process variable at the monitoring state, \emph{assuming complete data}. As the initial moments $\{a_{k}\}_{k=1}^{K}$ and the data moments $\{d_{k}\}_{k=1}^{K}$ are all the practitioner knows, the ITPC, a three-step procedure, constructs an information chart based on them, to detect whether a change has occurred in the process distribution between the monitoring and in-control states. 

\par The ITPC algorithm follows these steps:
\begin{enumerate}
\item Use the Maximum Entropy principle to estimate the unknown distribution $P_0$ of the process variable for the in-control state, $P^*$. In mathematical terms, obtain 
\begin{align*}
p^*=\argmax\left\{-\int_{\Omega}\log p\,dP:p\geq 0, \int_{\Omega} g_{k}\,dP=a_k,\, k=0,1,\ldots,K\right\},
\end{align*}
where $p^*$ is the density of $P^*$ with respect to Lebesgue measure on $\Omega$.
\item Use the Minimum Discrimination Information (MDI) principle to estimate the process variable's distribution at the monitoring state. Mathematically, this step combines $P^*$ and the data moments $\{d_{k}\}_{k=1}^{K}$ as inputs in the following relative entropy minimization problem,
\begin{equation}\label{eq - KL Min Problem Illust Example}
\begin{aligned}
\text{minimize}& \quad K(p,p^*)=\begin{cases} \int_{\Omega}p\log(p)\,dP^* &\mbox{if } p\geq0,\quad \int_{\Omega}p\,dP^*=1\\
\infty &\mbox{elsewhere},\end{cases} \\
\text{subject to}&\quad \int_{\Omega} pg_{k}\,dP^*=d_k,\, k=1,\ldots,K,\quad p\in L_1(P^*).
\end{aligned}
\end{equation}
whose solution is $p^{**}$ -- the $P^*$-density of $P^{**}$. This step can be applied with samples of the process variable at different monitoring states, $t=1,2,\ldots$, to produce a sequence $p_t^{**}$ for $t=1,2,\ldots$.  
\item Use the control function $I_t=2n_t\,K(p_t^{**},p^*)$ from the previous step to detect if a change has occurred in the distribution of the process variable between monitoring states and the in-control state, where $n_t$ is the sample size in state $t$. In particular, for each $t$. when $I_t$ is not ``sufficiently'' small, then the process is judged as not in-control. One can report the plot of $t\mapsto I_t$ with an upper control limit set as the $(1-\alpha)$ quantile of a distribution that indicates uncertainty about $I_t$. The uncertainty associated with $I_t$ may be accounted for either in the Bayesian paradigm or in the frequentist paradigm. 
\end{enumerate} 
As the steps above show, this algorithm hinges upon the complete data assumption at the monitoring phase. In many applications, sampled data can be incomplete due to censoring, so one must decide how to treat this problem in practice. Next, we propose a modification of the ITPC algorithm that extends its scope of applicability to such scenarios. 

\par Suppose that the data collected at the monitoring stage in state $t$ are realizations of a random sample from a random interval $Y=\left[\underline{Y},\overline{Y}\right]\subset \Omega\subset\mathbb{R}$, where the random vector $(\underline{Y},\overline{Y})$ has an unknown distribution, $\mathbb{P}$. Consequently, the MDI principle in Step 2 of the ITPC algorithm cannot be implemented, as the data are incomplete. The set of all distributions on $(\Omega,\mathcal{F})$ induced by measurable selections, $Z$, of $Y$, i.e., $Z(\omega)\in Y(\omega)$ almost surely $\mathbb{P}$ helps to model the censored data without introducing any structure on the nature of the censoring process If $P^*$ is not a distribution of a measurable selection of $Y$, then it cannot be compatible with the random interval, as there would be positive probability on the event $\left\{\omega\in\Omega:\,Z(\omega)\not\in Y(\omega)\right\}$, and hence, the process would be judged as \emph{not} in-control.

\par A fundamental result due to~\citet{Artstein1983DistributionsOR} characterizes the set of selectionable distributions in a general setting. A key ingredient in the characterization is the capacity functional of $Y$, which fully describes its stochastic behavior, given by the map 
\begin{align*}
\sigma\left([a,b]\right) & =\text{Prob}\left[\underline{Y}<a,\overline{Y}>a\right]+\text{Prob}\left[\underline{Y}\in[a,b]\right],
\end{align*}
whose domain is $\mathcal{C}(\Omega)=\left\{[a,b]\subset\Omega: a,b\in\Omega\right\}$. Artstein's result applied to this example shows that this set consists of all the distributions on $(\Omega,\mathcal{F})$ having $P^*$-density, $p$, satisfying the inequalities
\begin{align}\label{eq - Inequality restrictions}
\int_{\Omega}1\left[\omega\in C\right]p\,dP^*\leq \sigma\left(C\right)\quad\forall C\in \mathcal{C}(\Omega).
\end{align}
If the practitioner does not know $\sigma\left(\cdot\right)$, then it can be estimated using the interval data $Y_1,\ldots,Y_n$. A natural estimator of it is its empirical variant $C\mapsto\hat{\sigma}(C)=\frac{1}{n_t}\sum_{i=1}^{n_t}1\left[C\cap Y_i\neq\emptyset\right]$.
Hence, in Step 2, we can replace the optimization problem~(\ref{eq - KL Min Problem Illust Example}) with 
\begin{equation}\label{eq - KL Min Problem-1}
\begin{aligned}
\text{minimize}& \quad K(p,p^*)=\begin{cases} \int_{\Omega}p\log(p)\,dP^* &\mbox{if } p\geq0,\quad \int_{\Omega}p\,dP^*=1\\
\infty &\mbox{elsewhere},\end{cases} \\
\text{subject to}&\quad\int_{\Omega}\left(1\left[\omega\in C\right]-\hat{\sigma}\left(C\right)\right)\, p\,dP^*\leq 0\;\forall C\in \mathcal{C}(\Omega),\quad p\in L_1(P^*),
\end{aligned}
\end{equation}
Observe that the optimization problem~(\ref{eq - KL Min Problem-1}) is a special case of the one we consider in the manuscript, with $Q=P^*$, $m(p)=K(p,p^{*})$, $\Gamma=\mathcal{C}(\Omega)$ and $f_{\gamma}(\omega)=1\left[\omega\in C\right]-\hat{\sigma}\left(C\right)$ for $\gamma=C\in \mathcal{C}(\Omega)$.

\par The ITPC framework can be modified by replacing the relative entropy problem~(\ref{eq - KL Min Problem Illust Example}) in its second step with the one in~(\ref{eq - KL Min Problem-1}), and then build the information chart as in the third step but now using the optimal value $K(p,p^*)$ in~(\ref{eq - KL Min Problem-1}). This third step entails calculating quantiles of a distribution that indicates uncertainty about $I_t$, where there are now infinitely many moment inequality restrictions. If the practical objective is to detect when the process is \emph{not} in-control in state $t$, then a frequentist approach to determine the appropriate $(1-\alpha)$ quantile could be based on the null distribution of $I_t=2n_t\,K(p_t^{**},p^*)$ for the hypothesis testing problem:
\begin{align*}
H_0: \int_{\Omega}\left(1\left[\omega\in C\right]-\sigma\left(C\right)\right)\,dP^*\leq 0\;\forall C\in \mathcal{C}(\Omega)\quad\text{Vs.}\quad H_1:\,\text{not}\,H_0,
\end{align*}
where $\alpha$ is the significance level. Here, the source of uncertainty in the MDI step is due to the sampling variability of the random intervals $\{Y_1,\ldots,Y_{n_t}\}$.  One can also consider a Bayesian approach where the uncertainty associated with $I_t$ will be accounted for via the prior and posterior distributions of the moment vectors $a_1,\ldots,a_K$. In general, this approach ought to be determined by Monte Carlo methods. Developing these statistical procedures and their properties goes beyond the intended scope of the paper and is left for future research.

\section{Setup and Preliminaries}\label{Section - Setup}

\par We first introduce some notation and useful facts from functional analysis (which can be found in textbooks, such as~\citet{Rudin-Book},~\citet{Folland}, and~\citet{Bogachev}). Throughout this paper, the reference PD is $Q$, which is defined on the measurable space $(\Omega,\mathcal{F})$. Let $\mathcal{B}(\mathbb{R})$ denote the Borel $\sigma$-algebra on $\mathbb{R}$, and recall that $\mathcal{F}$ is a $\sigma$-algebra on $\Omega$. For any real-valued function that is $\mathcal{F}/\mathcal{B}(\mathbb{R})$
measurable, we define its integral with respect to $Q$ as $\int_{\Omega}z\,dQ=\int_{\Omega}z^+\,dQ-\int_{\Omega}z^-\,dQ$, provided that at least one of the integrals on the right side of this equality is finite, where
$$z^+(\omega)=\begin{cases} z(\omega) & \mbox{if }0\leq z(\omega)\leq\infty\\
0 &\mbox{otherwise}\end{cases}\quad \text{and}\quad z^-(\omega)=\begin{cases} -z(\omega) & \mbox{if }-\infty\leq z(\omega)\leq0\\
0 &\mbox{otherwise}\end{cases}$$ for each $\omega\in\Omega$. This is the setup we are mainly concerned with regarding the definition of all Lebesgue integrals in this paper. Now the set of real-valued functions $\{f_\gamma:\gamma\in\Gamma\}$ that define the inequality restrictions in~(\ref{eq - KL Min Problem}) are moment functions, which means $f_\gamma$ is measurable $\mathcal{F}/\mathcal{B}(\mathbb{R})$ for each $\gamma\in\Gamma$. Furthermore, the integrals in~(\ref{eq - KL Min Problem}) are well-defined in the sense described above.

\par We define for $r\geq 1$ the Lebesgue space
\begin{align*}
L_r(Q)=\left\{h:\Omega\rightarrow\mathbb{R}:\,\text{$h$ is measurable $\mathcal{F}/\mathcal{B}(\mathbb{R})$ and}\;\left(\int_{\Omega}|h|^r\,dQ\right)^{\frac{1}{r}}<\infty\right\}.
\end{align*}
These Lebesgue spaces have respective norms $\|h\|_{L_r(Q)}=\left(\int_{\Omega}|h|^r\,dQ\right)^{\frac{1}{r}}$, and are Banach spaces when they are equipped with their norm topology. The results of this paper focus on the space $L_1(Q)$. Its topological dual, $L_1(Q)^*$, is the vector space whose elements are the continuous linear functionals on $L_1(Q)$, plays a central role in the definition of the Gelfand-Pettis vector-valued integral. Because $(\Omega,\mathcal{F},Q)$ is a probability space, $L_1(Q)^*$ is isometrically isomorphic to the Lebesgue space $L_\infty(Q)$, where $$\|h\|_{L_\infty(Q)}=\inf\left\{a\geq0:Q\left(\{\omega\in\Omega:|h(\omega)|>a\}\right)=0\right\}.$$ Whence, $\Lambda\in L_1(Q)^*$ can be identified with $h\in L_\infty(Q)$ through the map $\Lambda_h:X\rightarrow \int_{\Omega}x\,h\,dQ$. The Lebesgue space $L_\infty(Q)$ is also a Banach space when it is equipped with the norm $\|\cdot\|_{L_\infty(Q)}$.

\par The Gelfand-Pettis vector-valued integral in this paper's setup is defined for \emph{any} $X\subset L_1(Q)$ that is compact in the norm topology. Suppose that $\xi$ is a positive Radon measure on the measure space $(X,\mathcal{B}(X))$, where $\mathcal{B}(X)$ is the Borel $\sigma$-algebra of $X$.
\begin{definition}[Gelfand-Pettis Integral]\label{Def - GF integral}
 Suppose that the scalar functions defined on $X$ given by $x\mapsto\Lambda(x)$ are integrable with respect to $\xi$, for every $\Lambda\in L_1(Q)^*$. If there exists a vector $y\in X$ such that
 \begin{align*}
\Lambda(y)=\int_{X}\Lambda(x)\,d\xi(x)\quad\forall \Lambda\in L_1(Q)^*,
\end{align*}
then we define $\int_{X}x\,d\xi(x)=y$.
\end{definition}
As $L_1(Q)^*$ separates points on $L_1(Q)$, there is at most one such $y$ that satisfies Definition~\ref{Def - GF integral}. Thus, there is no uniqueness problem. The existence of $y$ follows from an application of Theorem~3.27 of~\citet{Rudin-Book}, because (a) $X$ is compact in the $L_1(Q)$-norm, and (b) $L_1(Q)$ with its norm topology is a Fr\'{e}chet space.\footnote{See Part (f) in Section~1.8 of \citet{Rudin-Book} for the definition of Fr\'{e}chet space.} Denote by $C(X)$ the Banach space of all real-valued continuous functions on $X$. The Riesz Representation Theorem identifies the topological dual space $C(X)^*$ with the space of all real Radon measures on $X$. The subset of $C(X)^*$ corresponding to positive measures serves as the space that we use to model the Radon measure $\xi$ in the definition of the Gelfand-Pettis integral.

\par Throughout the paper, we use the following notations. For a sequence $\{y_n\}_{n\geq1}\subset L_1(Q)$ the notation, $y_n\stackrel{L_1(Q)}{\longrightarrow}y$, $y_n\stackrel{Q}{\longrightarrow}y$ and $y_n\stackrel{w}{\longrightarrow}y$, denote convergence of the sequence $\{y_n\}_{n\geq1}$ to $y$ in the norm $\|\cdot\|_{L_1(Q)}$, the topology of convergence in $Q$ measure, and in the weak topology of $L_1(Q)$, respectively. Additionally, for $\{\xi_n\}_{n\geq1}\subset C(X)^*$, the notation
$\xi_n\stackrel{w*}{\longrightarrow}\xi$ means the sequence $\{\xi_n\}_{n\geq1}$ converges to $\xi$ in the weak-star topology of $C(X)^*$. For any subset $X$ of a vector space, we denote the set of extreme points of $X$ by $\text{ex}\left(X\right)$. For any two PDs, $P_1$ and $P_2$, we write $P_1\ll P_2$ to denote that $P_1$ is absolutely continuous with respect to $P_2$. For any $y\in L_1(Q)$, $\delta_{y}$ denotes the Dirac delta function at $y$.

\section{Main Results}\label{Section- Main Results}
This section presents the main results of the paper. We shall describe $I$-projections onto a set of PDs defined by moment inequality restrictions, given by
\begin{align}\label{eq - constraint set M}
\mathcal{M}=\left\{p\in L_{1}(Q):m(p)<\infty,\;\int_{\Omega}f_\gamma p\,dQ\leq 0\;\forall\gamma\in\Gamma \right\},
\end{align}
where $m(\cdot)$ is the objective function in~(\ref{eq - KL Min Problem}). Note that the condition $m(p)<\infty$ imposes the restrictions for $p\in\mathcal{M}$ to be a $Q$-density function. We use the general Fenchel duality approach of \citet{Bhatacharya-Dykstra} to develop a dual optimization problem that is equivalent to the stated $I$-projection problem, 
\begin{align}\label{eq - primal Problem}
\inf\left\{m(p),\,p\in\mathcal{M}\right\}.
\end{align}
A significant advancement in their approach is the use of $L_0(Q)$ -- the vector space of measurable extended real-valued functions on the probability space $(\Omega,\mathcal{F},Q)$ -- as a general dual space. This contrasts with the standard dual space, $L_\infty(Q)$, which may prove restrictive because $\{f_\gamma:\gamma\in\Gamma\}\not\subset L_\infty(Q)$ can hold in practice, resulting in a failure of the standard duality approach.\footnote{See Example 4.1(a) of~\citet{Bhatacharya-Dykstra} for a simple but insightful illustration of their point with a single inequality restriction.} Of course, the standard duality approach would be successful if $\{f_\gamma:\gamma\in\Gamma\}\subset L_\infty(Q)$, holds; however, this structure on the moment functions can be restrictive in practice.

\par The main tool we use is Theorem 2.2 in~\citet{Bhatacharya-Dykstra}, which is based on the positive conjugate cone, $\mathcal{M}^{\oplus}=\left\{y\in L_0(Q): \int_{\Omega}y\,p\,dQ\geq0\;\forall p\in\mathcal{M}\right\}$. Let $\mathcal{D}\subset\mathcal{M}^{\oplus}$, then their result asserts that the pair, $p_Q\in\mathcal{M}$ and $y_0\in\mathcal{D}$, are respective solutions of~(\ref{eq - primal Problem}) and
\begin{align}\label{eq - Fenchel Dual Problem}
\inf\left\{\int_{\Omega}e^{y}\,dQ;\; y\in\mathcal{D}\right\},
\end{align}
if $\log\left(\int_{\Omega}e^{y_0}\,dQ\right)+m(p_Q)\leq 0$. Therefore, to use their result, we must specify $\mathcal{D}$ and $p_Q$ to ensure existence of a solution in the problem~(\ref{eq - Fenchel Dual Problem}) and satisfy the specified inequality constraint. Accordingly, this paper puts forward a specification of $\mathcal{D}$ and conditions on $\{f_\gamma: \gamma\in\Gamma\}$ and $Q$ that guarantees existence of a unique solution, $y_0$, for the dual problem~(\ref{eq - Fenchel Dual Problem}), where $p_Q=e^{y_0}/\int_{\Omega}e^{y_0}\,dQ$ is the $I$-projection.

\par In light of the form of $\mathcal{M}^{\oplus}$, we consider the dual problem~(\ref{eq - Fenchel Dual Problem}) on the following domain
\begin{align}\label{eq - domain}
\mathcal{D}=\left\{y\in\mathcal{M}^{\oplus}: y\in\overline{\text{co}}(\overline{\mathcal{V}})\cdot\alpha,\;\alpha\geq0\right\},
 \end{align}
where $\overline{\mathcal{V}}$ is the $L_1(Q)$-norm closure of $\mathcal{V}$, $\overline{\text{co}}(\overline{\mathcal{V}})$ is the closed convex hull of
$\overline{\mathcal{V}}$ in the $L_1(Q)$-norm, and
\begin{align}\label{eq - moment fxns set}
\mathcal{V}= \{-f_\gamma:\gamma\in\Gamma\}.
\end{align}
The following assumptions on $Q$ and the set of moment functions $\mathcal{V}$ are helpful for developing properties of the problem~(\ref{eq - Fenchel Dual Problem}) on the domain~(\ref{eq - domain}):
\begin{assumption}\label{Assumption - Fenchel}
(i) there exists $\gamma\in\Gamma$ and $\alpha>0$ such that $\int_{\Omega}e^{-\alpha f_\gamma}\,dQ<\infty,$ (ii) $\sup_{\gamma\in\Gamma}\int_{\Omega}|f_\gamma|\,dQ<\infty$, and (iii)
 $Q\left(\omega\in\Omega:y(\omega)>0\quad\forall y\in\mathcal{V}\right)>0$.
\end{assumption}
and
\begin{assumption}\label{Assump - L_1 and total boundedness}
$\mathcal{V}$ in~(\ref{eq - moment fxns set}) is a precompact subset of $L_1(Q)$ in the norm topology.
\end{assumption}

\par Before presenting our formal results, we briefly discuss the implications of these assumptions for solving the paired problems~(\ref{eq - primal Problem}) and~(\ref{eq - Fenchel Dual Problem}), using Theorem~2.2 of \citet{Bhatacharya-Dykstra}. On solving the dual problem~(\ref{eq - Fenchel Dual Problem}), Part (i) of Assumption~\ref{Assumption - Fenchel} ensures that $\left\{y\in\mathcal{D}:\int_{\Omega}e^{y}\,dQ<+\infty\right\}\neq\emptyset$. Part (ii) of Assumption~\ref{Assumption - Fenchel} implies $\{f_{\gamma}:\gamma\in\Gamma\}$ is a uniformly bounded subset of $L_1(Q)$, and Part (iii) implies the existence of $p\in \mathcal{M}$ such that the moment inequality constraints hold with strict inequality. Together, Parts (ii) and (iii) of Assumption~\ref{Assumption - Fenchel} are sufficient conditions for the lower semi-continuity and coercivity of the objective function in~(\ref{eq - Fenchel Dual Problem}) with respect to the $L_1(Q)$-norm on the domain $\mathcal{D},$ which are the ingredients for establishing existence: $\arginf\left\{\int_{\Omega}e^{y}\,dQ: y\in \mathcal{D}\right\}\neq\emptyset$. As $\mathcal{D}$ is a convex set and the objective function is strictly convex on $\mathcal{D}$ (because of the strict convexity of the exponential function), the dual problem~(\ref{eq - Fenchel Dual Problem}) has a unique minimizer, $y_0$. Then, using the necessary first-order conditions of the dual problem,  we show $p_Q=e^{y_0}/\int_{\Omega}e^{y_0}\,dQ\in\mathcal{M}$, and hence, it solves the $I$-projection problem~(\ref{eq - primal Problem}) by Theorem~2.2 of~\citet{Bhatacharya-Dykstra}, because $\log\left(\int_{\Omega}e^{y_0}\,dQ\right)+m(p_Q)=0$, holds. Combining this result with Assumption~\ref{Assump - L_1 and total boundedness} establishes that this unique solution depends on the set of binding moments
\begin{align}\label{eq - Binding moment fxns set}
B=\left\{v\in\mathcal{V}:\int_{\Omega}v\,p_Q dQ=0\right\}.
\end{align}
\noindent Assumption~\ref{Assump - L_1 and total boundedness} is a mild restriction that is widely used in statistics and probability theory. For example, Glivenko-Cantelli and Donsker classes of moment functions satisfy precompactness conditions in terms of their metric entropy covering and/or bracketing numbers (e.g., see Theorems 2.4.1 and 2.5.2 in~\citet{VDV-W}), which end up either being equivalent to or imply precompactness of the class in the $L_1(Q)$-norm.

\par We now present the first result of this paper.
\begin{theorem}[Existence and Exponential Family Representation]\label{Thm - Fenchel}
Let the constraint set $\mathcal{M}$ be given by~(\ref{eq - constraint set M}), and let the sets $\mathcal{V}$ and $B$ be given by~(\ref{eq - moment fxns set}) and~(\ref{eq - Binding moment fxns set}), respectively. The following statements hold.
\begin{enumerate}
\item If Assumption~\ref{Assumption - Fenchel} holds, then $\arginf\left\{\int_{\Omega}e^{y}\,dQ: y\in \mathcal{D}\right\}\neq\emptyset$ and
\begin{align}\label{eq - Representation}
p_Q=\frac{e^{y_0}}{\int_{\Omega}e^{y_0}\,dQ},\quad \text{where}\quad y_0\equiv\arginf\left\{\int_{\Omega}e^{y}\,dQ: y\in \mathcal{D}\right\},
\end{align}
solves the $I$-projection problem~(\ref{eq - primal Problem}).
\item Suppose that $\exists \gamma\in\Gamma$ such that $\int_{\Omega}f_\gamma\,dQ>0$. If Assumptions~\ref{Assumption - Fenchel} and~\ref{Assump - L_1 and total boundedness} hold, then $y_0\in\overline{\text{span}_{+}(B)}$, where $\overline{\text{span}_{+}(B)}$ is the $L_1(Q)$-norm closure of the positive linear span of $B$.
\end{enumerate}
\end{theorem}

\begin{proof}
\par {\bf Part 1}. The proof proceeds by the direct method. To show under our assumptions that \\ $\arginf\left\{\int_{\Omega}e^{y}\,dQ: y\in \mathcal{D}\right\}\neq\emptyset$ (i.e., existence of a minimizer), holds, it is sufficient to establish that the map $y\mapsto \int_{\Omega}e^{y}\,dQ$ is lower semi-continuous and coercive on $\mathcal{D}$, with respect to the $L_1(Q)$-norm. The uniqueness of the minimizer (up to equivalence class, i.e., a.s.$-Q$) is implied by the strict convexity of this map on $\mathcal{D}$, which follows from the strict convexity of the exponential function and Part (i) of Assumption~\ref{Assumption - Fenchel}.

\par Toward that end, we first establish the lower semi-continuity of the map $y\mapsto \int_{\Omega}e^{y}\,dQ$. Suppose that $\left\{y_n\right\}_{n\geq1}\subset \mathcal{D}$ such that $y_n\stackrel{L_1(Q)}{\longrightarrow}y$, then $\exists \{n_k\}_{k\geq 1}$ such that $\displaystyle\lim_{k\rightarrow\infty} \int_{\Omega}e^{y_{n_k}}\,dQ=\liminf_{n\rightarrow\infty}\int_{\Omega}e^{y_n}\,dQ$. Now we can take a further subsequence $\{n_{k_\ell}\}_{\ell\geq 1}$ such that $\displaystyle\lim_{\ell\rightarrow\infty}y_{n_{k_\ell}}(\omega)=y(\omega)$ a.s.$-Q$. This yields
\begin{align*}
\int_{\Omega}e^{y}\,dQ=\int_{\Omega}e^{\lim_{\ell\rightarrow\infty}y_{n_{k_\ell}}}\,dQ & \leq \liminf_{\ell\rightarrow\infty}\int_{\Omega}e^{y_{n_{k_\ell}}}\,dQ\;(\text{by Fatou's Lemma})\\
& = \lim_{k\rightarrow\infty} \int_{\Omega}e^{y_{n_k}}\,dQ\; (\text{as $\{n_{k_\ell}\}_{\ell\geq 1}\subset \{n_k\}_{k\geq 1}$}) \\
& = \liminf_{n\rightarrow\infty}\int_{\Omega}e^{y_n}\,dQ,
\end{align*}
establishing the desired result.

\par We now establish the map $y\mapsto \int_{\Omega}e^{y}\,dQ$ is norm-coercive on $\mathcal{D}$; that is, it satisfies
\begin{align}\label{eq - proof thm 2 00}
\int_{\Omega}e^{y}\,dQ\rightarrow\infty\quad \text{as}\quad \|y\|_{L_1(Q)}\rightarrow \infty,
\end{align}
on the set $\mathcal{D}$. Observe that $y\in\mathcal{D}$ implies $y=\alpha\cdot y^\prime$ for some $y^\prime\in\overline{\text{co}}(\overline{\mathcal{V}})$ and $\alpha\geq0$, and Part (ii) of Assumption~\ref{Assumption - Fenchel} implies
$\|y^\prime\|_{L_1(Q)}\leq \sup_{\gamma\in\Gamma}\int_{\Omega}|f_\gamma|\,dQ<\infty$ $\forall y^\prime\in\overline{\text{co}}(\overline{\mathcal{V}})$, which is a uniform bound. Consequently, $\|y\|_{L_1(Q)}\rightarrow \infty$ only when $\alpha\rightarrow \infty$, for $y\in\mathcal{D}$. To connect this structure to the aforementioned map,
\begin{align}
\int_{\Omega}e^{y}\,dQ & = \int_{\Omega}e^{y}\,1\left[e^{y}\leq\|y\|_{L_1(Q)} \right]\,dQ+\int_{\Omega}e^{y}\,1\left[e^{y}>\|y\|_{L_1(Q)} \right]\,dQ\nonumber\\
& \geq \|y\|_{L_1(Q)}\,Q\left(e^{y}>\|y\|_{L_1(Q)} \right)\nonumber\\
& = \|y\|_{L_1(Q)}\,Q\left(y^\prime>\frac{\log\alpha+\log \|y^\prime\|_{L_1(Q)}}{\alpha}\right)\nonumber\\
&\geq  \|y\|_{L_1(Q)}\,Q\left(y^\prime>\frac{\log\alpha+\log \sup_{\gamma\in\Gamma}\int_{\Omega}|f_\gamma|\,dQ}{\alpha}\right)\nonumber\\
&\geq  \|y\|_{L_1(Q)}\,Q\left(y^\prime>\frac{\log\alpha+\log \sup_{\gamma\in\Gamma}\int_{\Omega}|f_\gamma|\,dQ}{\alpha}\,\,\forall y^\prime\in\overline{\text{co}}(\overline{\mathcal{V}})\right)\label{eq - proof thm 2 000}
\end{align}
Now taking limits of both sides of the inequality in~(\ref{eq - proof thm 2 000}):
\begin{align*}
\lim_{\|y\|_{L_1(Q)}\rightarrow \infty}\int_{\Omega}e^{y}\,dQ & = \infty\,\lim_{\|y\|_{L_1(Q)}\rightarrow \infty}Q\left(y^\prime>\frac{\log\alpha+\log \sup_{\gamma\in\Gamma}\int_{\Omega}|f_\gamma|\,dQ}{\alpha}\,\,\forall y^\prime\in\overline{\text{co}}(\overline{\mathcal{V}})\right)\\
& = \infty\,\lim_{\alpha\rightarrow \infty}Q\left(y^\prime>\frac{\log\alpha+\log \sup_{\gamma\in\Gamma}\int_{\Omega}|f_\gamma|\,dQ}{\alpha}\,\,\forall y^\prime\in\overline{\text{co}}(\overline{\mathcal{V}})\right)\\
& = \infty\,Q\left(y^\prime>0\,\,\forall y^\prime\in\overline{\text{co}}(\overline{\mathcal{V}})\right)=\infty,
\end{align*}
since $\sup_{\gamma\in\Gamma}\int_{\Omega}|f_\gamma|\,dQ<\infty$ and $Q\left(y^\prime>0\,\,\forall y^\prime\in\overline{\text{co}}(\overline{\mathcal{V}})\right)>0$ by Conditions (ii) and (iii) of Assumption~\ref{Assumption - Fenchel}, respectively. Therefore (\ref{eq - proof thm 2 00}) holds, as desired.

\par As Part (i) of Assumption~\ref{Assumption - Fenchel} implies that $\mathcal{D}\neq\emptyset$, the arguments above establish that the set of minimizers is nonempty, i.e., $\arginf\left\{\int_{\Omega}e^{y}\,dQ: y\in \mathcal{D}\right\}\neq\emptyset$, holds. Now we shall establish the uniqueness a.s.$-Q$ of the minimizer. Part (i) of Assumption~\ref{Assumption - Fenchel} implies that the minimizers cannot be where the objective function equals $\infty.$ Combining this implication with the strict convexity of the map $y\mapsto \int_{\Omega}e^{y}\,dQ$ on $\mathcal{D}$, implies that there is a unique minimizer (up to equivalence class). Let $\beta\in(0,1)$ and $y_1,y_2\in \mathcal{D}$ such that $y_1\neq y_2$ holds as equivalence classes. Additionally, let $y_3=\beta\, y_1+(1-\beta)\,y_2$. Then $\int_{\Omega}e^{y_3}\,dQ< \beta\int_{\Omega}e^{y_1}\,dQ+(1-\beta)\int_{\Omega}e^{y_3}\,dQ$, holds, by the strict convexity of the exponential function. This establishes the strict convexity of the map, and hence, the set of minimizers $\arginf\left\{\int_{\Omega}e^{y}\,dQ: y\in \mathcal{D}\right\}$, is unique up to equivalence class.

\par Next, we develop the representation of the $I$-projection's $Q$-density. From the above arguments let $y_0=\arginf\left\{\int_{\Omega}e^{y}\,dQ: y\in \mathcal{D}\right\}$. The set $\mathcal{D}$ is convex, and the objective function $g(y)=\int_{\Omega}e^{y}\,dQ$ is G\^{a}teaux differentiable, then by Theorem~2 on page 178 of~\citet{Luenberger}, $\frac{d}{dt}g\left(y_0+t(y-y_0)\right)\mid_{t=0}\,\geq0$ $\forall y\in\mathcal{D}$,
yielding $\int_{\Omega}(y-y_0)e^{y_0}\,dQ\geq0$ $\forall y\in\mathcal{D}$. By choosing $y=cy_0$ first with $c>1$ and then with $c<1$ (since $\mathcal{D}$ is also a cone), we obtain
\begin{align}\label{eq - proof thm 2 0}
\int_{\Omega}y_0e^{y_0}\,dQ=0,\;\text{and}\quad \int_{\Omega}ye^{y_0}\,dQ\geq0\;\forall y\in\mathcal{D}.
\end{align}
Let $p_Q=\frac{e^{y_0}}{\int_{\Omega}e^{y_0}\,dQ}$, and note that the second part of~(\ref{eq - proof thm 2 0}) implies that $\int_{\Omega}v\,p_Q\,dQ\geq0\quad\forall v\in\mathcal{V}$; hence, $p_Q\in\mathcal{M}$. Furthermore, $m(p_Q)+\log\left(\int_{\Omega}e^{y_0}\,dQ\right)=\int_{\Omega}y_0e^{y_0}\,dQ=0$, holds, by the first part of~(\ref{eq - proof thm 2 0}). Hence, by Theorem~2.2 of~\citet{Bhatacharya-Dykstra}, $p_Q$ solves the $I$-projection problem.

\par {\bf Part 2}. The proof proceeds by the direct method. We must establish that $y_0\in\overline{\text{span}_+(B)}$, holds. Since $Q$ violates the moments inequality restrictions, $y_0\in\mathcal{D}$ implies that $y_0=\alpha_0\,y^\prime_0$ with $\alpha_0>0$ and $y^\prime_0\in \overline{\text{co}}(\overline{\mathcal{V}})$. The result of Lemma~\ref{Lemma - KMT and MT} implies that there are only two cases to consider in establishing the desired result: (i) $y^\prime_0\in\text{ex}\left(\overline{\text{co}}(\overline{\mathcal{V}})\right)$, and (ii) $y^\prime_0\not\in\text{ex}\left(\overline{\text{co}}(\overline{\mathcal{V}})\right)$.

\par Starting with case (i), since $\text{ex}\left(\overline{\text{co}}(\overline{\mathcal{V}})\right)\subset\overline{\mathcal{V}}$ also by Lemma~\ref{Lemma - KMT and MT}, it must be that $y^\prime_0\in \overline{\mathcal{V}}$. Hence, either $y^\prime_0\in\mathcal{V}$ or $y^\prime_0\in\mathcal{V}^{\prime}$ (i.e., a limit point of $\mathcal{V}$ in the $L_1(Q)$-norm). If $y^\prime_0\in\mathcal{V}$, then 
\begin{align}\label{eq - proof thm 2 1 0}
\int_{\Omega}y_0e^{y_0}\,dQ=0\iff \int_{\Omega}y_0\,dP_Q=0\iff \alpha_0\int_{\Omega}y^\prime_0\,dP_Q=0,
\end{align}
holds. Now by the second part of~(\ref{eq - proof thm 2 0}) and that $\alpha_0>0$, we deduce that $\int_{\Omega}y^\prime_0\,dP_Q=0$. Consequently, $y^\prime_0\in B$, and hence, $y_0=\alpha_0\,y^\prime_0\in\overline{\text{span}_+(B)}$. Now suppose that $y^\prime_0\in\mathcal{V}^{\prime}$. As $\mathcal{V}^{\prime}\subset\overline{\text{co}}(\overline{\mathcal{V}})$, the equalities~(\ref{eq - proof thm 2 1 0}) must also hold in this case. Since $\alpha_0>0$, it follows that $\int_{\Omega}y^\prime_0\,dP_Q=0$, implying $y^\prime_0\in \overline{B}$, and hence, $y_0=\alpha_0\,y^\prime_0\in\overline{\text{span}_+(B)}$ as $\overline{B}\subset\overline{\text{span}_+(B)}$.

\par Next, consider case (ii): $y^\prime_0\not\in\text{ex}\left(\overline{\text{co}}(\overline{\mathcal{V}})\right)$. Then, $\exists n\in\mathbb{Z}_+$, $p_i>0$ for each $i=1,\ldots,n$ such that $\sum_{i=1}^{n}p_i=1$, for which $y_0=\alpha_0 \sum_{i=1}^{n}p_iv_i$, where $\alpha_0>0$ and $\left\{v_1,\ldots,v_n\right\}\subset\text{ex}\left(\overline{\text{co}}(\overline{\mathcal{V}})\right)\subset\overline{\mathcal{V}}$. Consequently,
\begin{align}\label{eq - proof thm 2 1}
\int_{\Omega}y_0e^{y_0}\,dQ=0\iff \int_{\Omega}y_0\,dP_Q=0\iff \alpha_0\sum_{i=1}^{n}p_i\int_{\Omega}v_i\,dP_Q=0,
\end{align}
holds. Now by the second part of~(\ref{eq - proof thm 2 0}) and that $\alpha_0,p_1,\ldots,p_n>0$, the equality~(\ref{eq - proof thm 2 1}) forces $\int_{\Omega}v_i\,dP_Q=0$ for each $i$. Finally, since $\text{ex}\left(\overline{\text{co}}(\overline{\mathcal{V}})\right)\subset \overline{\mathcal{V}}$, we must have that $\{v_1,\ldots,v_n\}\subset\overline{B}$, and hence, $y_0\in\overline{\text{span}_+(B)}$.
\end{proof}

\par The first result of Theorem~\ref{Thm - Fenchel} shows the $I$-projection is a member of the exponential family of distributions under Assumption~\ref{Assumption - Fenchel}. Furthermore, the representation yields the duality relationship
\begin{align}\label{eq - duality}
m(p_Q)=-\log\left(\int_{\Omega}e^{y_0}\,dQ\right).
\end{align}
The result of Part 2 of Theorem~\ref{Thm - Fenchel} indicates that the unique minimizer of the dual problem, $y_0$, is either a finite linear combination of the moment functions in $B$ with positive coefficients, or is a limit point of such combinations in the $L_1(Q)$-norm. Assumption~\ref{Assump - L_1 and total boundedness} is critical to obtaining this result. 

\par The computability of $y_0$ in~(\ref{eq - Representation}) is key to adopting the proposed duality approach in practice. Towards that end, we develop a reparametrization of the dual problem~(\ref{eq - Fenchel Dual Problem}) that enables the computation of $y_0$ using a simple approximation scheme. The reparametrization is derived from a representation of the elements of $\mathcal{D}$ in terms of the integrator in a weak vector-valued integral. Under Assumption~\ref{Assump - L_1 and total boundedness}, the set $\overline{\mathcal{V}}$ is compact in the $L_1(Q)$-norm. Furthermore, since $L_1(Q)$ with its norm topology is also Fr\'{e}chet space, the subset of $L_1(Q)$ given by $\overline{\text{co}}(\overline{\mathcal{V}})$ is also compact in the norm topology. Denote by $\mathcal{P}\subset C\left(\overline{\mathcal{V}}\right)^{*}$ the set of Radon probability measures on the set $\overline{\mathcal{V}}$. An application of Theorem~3.28 in~\citet{Rudin-Book} yields
\begin{align}\label{eq - Result of Rudin}
y\in\overline{\text{co}}(\overline{\mathcal{V}})\iff \exists\mu\in\mathcal{P}\;\text{such that}\;y=\int_{\overline{\mathcal{V}}}v\,d\mu(v),
\end{align}
where the integral in~(\ref{eq - Result of Rudin}) is to be understood as in Definition~\ref{Def - GF integral}. The characterization~(\ref{eq - Result of Rudin}) is the building block of the reparametrization, as
\begin{align}\label{eq - Result of Rudin Extended}
y\in\mathcal{D}\iff \exists\mu\in\mathcal{P}\,\text{and}\,\alpha\geq0\;\text{such that}\;y=\alpha\,\int_{\overline{\mathcal{V}}}v\,d\mu(v).
\end{align}
Next, define $\Xi\subset C\left(\overline{\mathcal{V}}\right)^{*}$ as the set of all positive Radon measures on $\overline{\mathcal{V}}$, and consider the following set
\begin{align}\label{eq - Upsilon set}
\Upsilon=\left\{\xi\in\Xi: \xi=\alpha\cdot\mu,\alpha\geq0\,\text{and}\,\mu\in\mathcal{P}\right\}.
\end{align}
The dual problem~(\ref{eq - Fenchel Dual Problem}) can now be reparametrized as
\begin{align}\label{eq - Reparametrized Fenchel Dual Problem}
\inf\left\{\int_{\Omega}e^{\int_{\overline{\mathcal{V}}}v\,d\xi(v)}\,dQ:\xi\in\Upsilon\right\}.
\end{align}

\par In the above framework, the following theorem establishes the existence of a solution to~(\ref{eq - Reparametrized Fenchel Dual Problem}).
\begin{theorem}[Reparametrization of the dual]\label{Thm - existence reparametrization}
Suppose that Assumptions~\ref{Assumption - Fenchel} and~\ref{Assump - L_1 and total boundedness}, hold. Let $\Upsilon$ be given by~(\ref{eq - Upsilon set}). Then  $\arginf\left\{\int_{\Omega}e^{\int_{\overline{\mathcal{V}}}v\,d\xi(v)}\,dQ:\xi\in\Upsilon\right\}\neq\emptyset$ and $$\arginf\left\{\int_{\Omega}e^{\int_{\overline{\mathcal{V}}}v\,d\xi(v)}\,dQ:\xi\in\Upsilon\right\}=\left\{\xi\in\Upsilon: y_0=\int_{\overline{\mathcal{V}}}v\,d\xi(v) \right\},$$
hold, where $y_0$ is given by~(\ref{eq - Representation}).
\end{theorem}
\begin{proof}
 The proof proceeds by the direct method. It follows steps similar to that in the proof of Part 1 of Theorem~\ref{Thm - Fenchel}. The goal is to establish under our assumptions that $\arginf\left\{ \int_{\Omega}e^{\int_{\overline{\mathcal{V}}}v\,d\xi(v)}\,dQ: \xi\in\Upsilon\right\}\neq\emptyset$ (i.e., existence of a minimizer), holds. As Part (i) of Assumption~\ref{Assumption - Fenchel} implies that $\Upsilon\neq\emptyset$, it is sufficient to establish that the map $\xi\mapsto \int_{\Omega}e^{\int_{\overline{\mathcal{V}}}v\,d\xi(v)}\,dQ$ is lower semi-continuous and coercive with respect to the weak-star topology of $C\left(\overline{\mathcal{V}}\right)^*$ on $\Upsilon$.

\par Toward that end, we first prove that this map is weak-star lower semi-continuous on $\Upsilon$. Let $\{\xi_n\}_{n\geq1}\subset\Upsilon$ be such that $\xi_{n}\stackrel{w^*}{\longrightarrow}\xi$, where $\xi\in\Upsilon$. Using the fact that the map $\xi\mapsto\int_{\overline{\mathcal{V}}}v\,d\xi$ is continuous when $C\left(\overline{\mathcal{V}}\right)^*$ is given the weak-star topology and $L_1(Q)$ has the weak topology, it follows that $\int_{\overline{\mathcal{V}}}v\,d\xi_{n}\stackrel{w}{\longrightarrow}\int_{\overline{\mathcal{V}}}v\,d\xi$. Since the exponential function is continuous on $\mathbb{R}$, it follows that $e^{\int_{\overline{\mathcal{V}}}v\,d\xi_{n}}\stackrel{w}{\longrightarrow} e^{\int_{\overline{\mathcal{V}}}v\,d\xi}$. Now by the Skorohod Representation Theorem (e.g., Theorem~7.2.14 in~\citet{Grimmet-Stirkazer}), there exists a probability space $(\Omega^\prime,\mathcal{F}^\prime,Q^\prime)$ and real-valued measurable functions $\{z_n\}_{n\geq 1}$ and $z$ on this probability space, such that
\begin{enumerate}[(a)]
\item $\{z_n\}_{n\geq1}$ and $z$ have the same probability distributions as $\{e^{\int_{\overline{\mathcal{V}}}v\,d\xi_{n}}\}_{n\geq 1}$ and $e^{\int_{\overline{\mathcal{V}}}v\,d\xi}$, respectively, and
\item $\lim_{n\rightarrow\infty}z_n(\omega^\prime)=z(\omega^\prime)$ a.s.$-Q^\prime$.
\end{enumerate}
One can set $\Omega^\prime=[0,1]$ with $\{z_n\}_{n\geq1}$ and $z$ being the quantile functions of $\{e^{\int_{\overline{\mathcal{V}}}v\,d\xi_{n}}\}_{n\geq 1}$ and $e^{\int_{\overline{\mathcal{V}}}v\,d\xi}$, respectively. Consequently, $z_n(\omega^{\prime}),z(\omega^{\prime})\geq 0$ for all $\omega^{\prime}\in [0,1]$. Therefore, by Fatou's Lemma,
\begin{align*}
\int_{\Omega}e^{\int_{\overline{\mathcal{V}}}v\,d\xi_0}\,dQ=\int_{[0,1]}z\,dQ^{\prime}\leq\liminf_{n\rightarrow\infty}\int_{[0,1]}z_n\,dQ^{\prime}=\liminf_{n\rightarrow\infty}\int_{\Omega}e^{\int_{\overline{\mathcal{V}}}v\,d\xi_{n}}\,dQ.
\end{align*} 

\par Next, we show the objective function is weak-star coercive on $\Upsilon$. This means establishing
\begin{align}\label{eq - proof thm 3 00}
\int_{\Omega}e^{\int_{\overline{\mathcal{V}}}v\,d\xi(v)}\,dQ\rightarrow\infty\quad \text{as}\quad \int_{\overline{\mathcal{V}}}g(v)\,d\xi(v)\rightarrow \pm\infty\;\forall g\in C\left(\overline{\mathcal{V}}\right),
\end{align}
where the latter divergence restricts $\xi\in\Upsilon$. Note that $\int_{\overline{\mathcal{V}}}g(v)\,d\xi(v)=\alpha\int_{\overline{\mathcal{V}}}g(v)\,d\mu(v)$ holds for every $\xi\in\Upsilon$, and $\left|\int_{\overline{\mathcal{V}}}g(v)\,d\mu(v)\right|\leq \int_{\overline{\mathcal{V}}}|g(v)|\,d\mu(v)\leq\sup_{v\in\overline{\mathcal{V}}}|g(v)| <\infty$ since $\overline{\mathcal{V}}$ is compact by Assumption~\ref{Assump - L_1 and total boundedness} and $\mu\in\mathcal{P}$. Therefore, divergences $\int_{\overline{\mathcal{V}}}g(v)\,d\xi(v)\rightarrow\pm\infty$ $\forall g\in C\left(\overline{\mathcal{V}}\right)$ arise only when $\alpha\rightarrow\infty$. To that end, by the tail sum representation of the objective function:
\begin{align*}
\int_{\Omega}e^{\int_{\overline{\mathcal{V}}}v\,d\xi(v)}\,dQ & =\int_{0}^{\infty}Q\left(\omega\in\Omega:e^{\int_{\overline{\mathcal{V}}}v\,d\xi(v)}>t\right)\,dt  \\ & =\int_{0}^{\infty}Q\left(\omega\in\Omega:\int_{\overline{\mathcal{V}}}v\,d\mu(v)>\frac{\log t}{\alpha}\right)\,dt  \\
& \geq \int_{1}^{\infty}Q\left(\omega\in\Omega:\int_{\overline{\mathcal{V}}}v\,d\mu(v)>\frac{\log t}{\alpha}\right)\,dt.
\end{align*}
Now from the characterization~(\ref{eq - Result of Rudin}), $\exists y^\prime\in\overline{\mathcal{V}}$ such that $y^\prime=\int_{\overline{\mathcal{V}}}v\,d\mu(v)$, so that
\begin{align*}
\int_{1}^{\infty}Q\left(\omega\in\Omega:\int_{\overline{\mathcal{V}}}v\,d\mu(v)>\frac{\log t}{\alpha}\right)\,dt& =\int_{1}^{\infty}Q\left(\omega\in\Omega: y^\prime>\frac{\log t}{\alpha}\right)\,dt, \\
& \geq\int_{1}^{\infty}Q\left(\omega\in\Omega:y^\prime>\frac{\log t}{\alpha}\,\forall y^\prime\in\overline{\text{co}}(\overline{\mathcal{V}})\right)\,dt.
\end{align*}

\par Now, let $h_\alpha(t)=Q\left(\omega\in\Omega:y^\prime>\frac{\log t}{\alpha}\,\forall y^\prime\in\overline{\text{co}}(\overline{\mathcal{V}})\right)$, and observe that $h_{\alpha_1}(\cdot)\leq h_{\alpha_2}(\cdot)$ for any $\alpha_1\leq \alpha_2$, and that $h_\alpha(\cdot)$ is measurable on $(\mathbb{R},\mathcal{B}\left(\mathbb{R}\right))$ for each $\alpha\geq0$. Furthermore, $h_\infty(t)=\lim_{\alpha\rightarrow\infty}h_\alpha(t)=Q\left(\omega\in\Omega:y^\prime>0\,\forall y^\prime\in\overline{\text{co}}(\overline{\mathcal{V}})\right)$, for each $t\geq1$, and this limit is positive by Condition (iii) of Assumption~\ref{Assumption - Fenchel}. Since $\int_{\overline{\mathcal{V}}}g(v)\,d\xi(v)\rightarrow\pm\infty$ $\forall g\in C\left(\mathcal{V}\right)$ implies $\alpha\rightarrow\infty$, it follows that $$\lim_{\alpha\rightarrow\infty}\int_{1}^{\infty}Q\left(\omega\in\Omega:y^\prime>\frac{\log t}{\alpha}\,\forall y^\prime\in\overline{\text{co}}(\overline{\mathcal{V}})\right)\,dt=\int_{1}^{\infty}Q\left(\omega\in\Omega:y^\prime>0\,\forall y^\prime\in\overline{\text{co}}(\overline{\mathcal{V}})\right)\,dt=\infty,$$ holds, by the Monotone Convergence Theorem (e.g., Theorem 2.14 in \citet{Folland}). Consequently, applying this limit to the above chain of inequalities, establishes that~(\ref{eq - proof thm 3 00}), holds.

\par Now we prove $\arginf\left\{\int_{\Omega}e^{\int_{\overline{\mathcal{V}}}v\,d\xi(v)}\,dQ:\xi\in\Upsilon\right\}=\left\{\xi\in\Upsilon: y_0=\int_{\overline{\mathcal{V}}}v\,d\xi(v) \right\}$. We proceed by contradiction with the direction: $$\xi^\prime\in\arginf\left\{\int_{\Omega}e^{\int_{\overline{\mathcal{V}}}v\,d\xi(v)}\,dQ:\xi\in\Upsilon\right\}\implies \xi^\prime\in\left\{\xi\in\Upsilon: y_0=\int_{\overline{\mathcal{V}}}v\,d\xi(v) \right\}.$$ Suppose $\xi^\prime\in\arginf\left\{\int_{\Omega}e^{\int_{\overline{\mathcal{V}}}v\,d\xi(v)}\,dQ:\xi\in\Upsilon\right\}$ and $\xi^\prime\notin \left\{\xi\in\Upsilon: y_0=\int_{\overline{\mathcal{V}}}v\,d\xi(v) \right\}$. Then, the first and second inclusions imply that
\begin{align}
\int_{\Omega}e^{\int_{\overline{\mathcal{V}}}v\,d\xi^\prime(v)}\,dQ\leq \int_{\Omega}e^{y_0}\,dQ\iff \int_{\Omega}e^{y^\prime}\,dQ\leq \int_{\Omega}e^{y_0}\,dQ,\label{eq - proof thm 3 6}
\end{align}
where $y^\prime=\int_{\overline{\mathcal{V}}}v\,d\xi^\prime(v)$. As Theorem~\ref{Thm - Fenchel} establishes that $y_0$ is the unique minimizer of~(\ref{eq - Fenchel Dual Problem}), the inequality~(\ref{eq - proof thm 3 6}) implies $y^\prime=y_0$; therefore
 $\xi^\prime\notin \left\{\xi\in\Upsilon: y_0=\int_{\overline{\mathcal{V}}}v\,d\xi(v) \right\}$ and  $\xi^\prime\in \left\{\xi\in\Upsilon: y_0=\int_{\overline{\mathcal{V}}}v\,d\xi(v) \right\}$, which is a contradiction.

\par Now we prove by contradiction the direction $$\xi^\prime\in\left\{\xi\in\Upsilon: y_0=\int_{\overline{\mathcal{V}}}v\,d\xi(v) \right\}\implies\xi^\prime\in\arginf\left\{\int_{\Omega}e^{\int_{\overline{\mathcal{V}}}v\,d\xi(v)}\,dQ:\xi\in\Upsilon\right\}.$$ Suppose $\xi^\prime\in\left\{\xi\in\Upsilon: y_0=\int_{\overline{\mathcal{V}}}v\,d\xi(v) \right\}$ and $\xi^\prime\notin\arginf\left\{\int_{\Omega}e^{\int_{\overline{\mathcal{V}}}v\,d\xi(v)}\,dQ:\xi\in\Upsilon\right\}$. Then these inclusions imply that $\int_{\Omega}e^{y_0}\,dQ=\int_{\Omega}e^{\int_{\overline{\mathcal{V}}}v\,d\xi^\prime(v)}\,dQ> \int_{\Omega}e^{\int_{\overline{\mathcal{V}}}v\,d\xi_1(v)}\,dQ$ for some \\ $\xi_1\in\arginf\left\{\int_{\Omega}e^{\int_{\overline{\mathcal{V}}}v\,d\xi(v)}\,dQ:\xi\in\Upsilon\right\}$. By~(\ref{eq - Result of Rudin Extended}), $\exists y_1\in\mathcal{D}$ such that $y_1=\int_{\overline{\mathcal{V}}}v\,d\xi_1(v)$, and hence, the above strict inequality becomes
\begin{align}\label{eq - proof thm 3 5}
\int_{\Omega}e^{y_0}\,dQ=\int_{\Omega}e^{\int_{\overline{\mathcal{V}}}v\,d\xi^\prime(v)}\,dQ> \int_{\Omega}e^{y_1}\,dQ.
\end{align}
The inequality~(\ref{eq - proof thm 3 5}) yields a contradiction as Theorem~\ref{Thm - Fenchel} implies that $\int_{\Omega}e^{y_0}\,dQ<\int_{\Omega}e^{y_1}\,dQ$ must also be true. This concludes the proof.
\end{proof}

\par Like the proof of Part 1 of Theorem~\ref{Thm - Fenchel}, we prove the existence result of Theorem~\ref{Thm - existence reparametrization} by establishing the lower semi-continuity and coercivity of the objective function
\begin{align}\label{eq - objective fxn reparametrized}
\xi\mapsto \int_{\Omega}e^{\int_{\overline{\mathcal{V}}}v\,d\xi(v)}\,dQ
\end{align}
 on the domain $\Upsilon$, but with respect to the weak-star topology of $C\left(\overline{\mathcal{V}}\right)^{*}$. Utilizing this topology on $C\left(\overline{\mathcal{V}}\right)^{*}$ is advantageous because
the Gelfand-Pettis integral $\int_{\overline{\mathcal{V}}}v\,d\xi(v)$, when identified as a mapping $\phi:C\left(\overline{\mathcal{V}}\right)^{*}\rightarrow L_1(Q)$, is continuous when $C\left(\overline{\mathcal{V}}\right)^{*}$ and $L_1(Q)$ are given their weak-star and weak topologies, respectively. The continuity of this mapping with those topologies facilitates the proof of lower semi-continuity of the objective function.

\par The result of Theorem~\ref{Thm - existence reparametrization} is that the reparametrized dual problem~(\ref{eq - Reparametrized Fenchel Dual Problem}) has a solution under the conditions of Theorem~\ref{Thm - Fenchel}. Furthermore, the solution set of this optimization problem is characterized by $y_0$ -- the solution of the Fenchel dual problem~(\ref{eq - Fenchel Dual Problem}), via the representation~(\ref{eq - Result of Rudin Extended}). It is important to note that unlike the dual problem~(\ref{eq - Fenchel Dual Problem}), the solution of the reparametrized dual problem~(\ref{eq - Reparametrized Fenchel Dual Problem}) may not be unique. This non-uniqueness arises from the non-uniqueness of the representation of $y\in\overline{\text{co}}(\overline{\mathcal{V}})$ in~(\ref{eq - Result of Rudin}), since it can be represented as the moment of different elements of $\mathcal{P}$. It is this feature of the representation that renders the objective function in~(\ref{eq - Reparametrized Fenchel Dual Problem}) as not necessarily being strictly convex on $\Upsilon$.

\par The advantage of the proposed reparametrization of the dual optimization problem is that it enables the computation of its optimal solutions using a simple approximation scheme. Since $L_1(Q)$ is a Banach space, and the identity map from $\overline{\mathcal{V}}$ into $L_1(Q)$ is continuous with respect to the $L_1(Q)$-norm on both spaces, Lemma~\ref{Lemma - Riemann Sum} shows the Gelfand-Pettis integral $\int_{\overline{\mathcal{V}}}h\,d\mu(h)$ can be approximated as the limit of ``Riemann sums'' in the $L_1(Q)$-norm. That is, given $\epsilon>0$, let $U=\left\{y\in L_1(Q): \|y\|_{L_1(Q)}\leq\epsilon\right\}$. Then, there corresponds a finite partition $\{E_i\}_{i=1}^{n}$ of $\overline{\mathcal{V}}$ such that
\begin{align}\label{eq - Riemann Sum Consruction}
\int_{\overline{\mathcal{V}}}v\,d\mu(v)-\sum_{i=1}^{n}\mu(E_i)v_i\in U\quad \forall v_i\in E_i,\;i=1\,\ldots,n.
\end{align}
Remarkably, the only structure on the partitions that is required for~(\ref{eq - Riemann Sum Consruction}) to hold is that for each $i$: $v-v^\prime\in U$ for all $v,v^\prime\in E_i$. Consequently, the partitions depend only on $\epsilon$, $Q$ and $\mathcal{V}$, and not $\mu$. This point is key to the approximation scheme put forward by this paper.

\par The sequence of discretizations of the infinite program~(\ref{eq - Reparametrized Fenchel Dual Problem}) we consider is indexed by $\{\epsilon_m\}_{m\geq1}\subset\mathbb{R}_{++}$ such that $\epsilon_m\downarrow0$. For each $m$, let $U_m=\left\{y\in L_1(Q): \|y\|_{L_1(Q)}\leq\epsilon_m\right\}.$ Then by Lemma~\ref{Lemma - Riemann Sum}, there corresponds a finite partition $\{E_{i,m}\}_{i=1}^{n_m}$ of $\overline{\mathcal{V}}$ with the property
\begin{align*}
\int_{\overline{\mathcal{V}}}v\,d\mu(v)-\sum_{i=1}^{n_m}\mu(E_{i,m})v_i\in U_m\quad \forall v_i\in E_{i,m},\;i=1\,\ldots,n_m.
\end{align*}
Under additional regularity conditions, the optimal value and minimizers of the infinite program~(\ref{eq - Reparametrized Fenchel Dual Problem}) can be approximated by the corresponding sequence of optimal values and minimizers of the finite programs as $m\rightarrow \infty$:
\begin{align}\label{eq - finite program}
\inf\left\{\int_{\Omega}e^{\alpha\sum_{i=1}^{n_m}\mu_i\,v_i}\,dQ;\;\alpha,\mu_i\geq0\;\forall i,\;\text{and}\;\sum_{i=1}^{n_m}\mu_i=1\right\},
\end{align}
where $v_i\in E_{i,m}$ are given for each $i$. These conditions are given by the following assumption
\begin{assumption}\label{Assump - Approx}
$\sup_{\gamma\in\Gamma}\int_{\Omega}e^{-\alpha f_\gamma}\,dQ<\infty\quad\forall \alpha>0$. 
\end{assumption}
\noindent The conditions of this assumption strengthens those in Condition (i) of Assumption~\ref{Assumption - Fenchel}. It implies the objective function in~(\ref{eq - finite program}) is well-defined over its domain for any choice of $v_i\in E_{i,m}$ for each $i$.

\par The next theorem formalizes our approximation results. 
\begin{theorem}[Discretization of the dual]\label{Thm - Computation}
Suppose the conditions of Theorem~\ref{Thm - existence reparametrization}, and Assumption~\ref{Assump - Approx}, hold. Let $\{\epsilon_m,U_m\}_{m\geq1}$ be described as above. For each $m$, let
$\{E_{i,m}\}_{i=1}^{n_m}$ be a partition of $\overline{\mathcal{V}}$ such that, for each $i$, $v-v^\prime\in U_m$ holds for all $v,v^\prime\in E_{i,m}$. Then, for each $m$,
\begin{align}\label{eq - existence of approximants}
\arginf\left\{\int_{\Omega}e^{\alpha\sum_{i=1}^{n_m}\mu_i\,v_i}\,dQ;\;\alpha,\mu_i\geq0\;\forall i,\;\text{and}\;\sum_{i=1}^{n_m}\mu_i=1\right\}\neq\emptyset,
\end{align}
for any $v_i\in E_{i,m}$ where $i=1,\ldots,n_m$. Furthermore, for each $m$ and $v_i\in E_{i,m}$ with $i=1,\ldots,n_m$, define the corresponding Radon measure $\xi_m=\alpha_{n_m}\sum_{i=1}^{n_m}\mu_{i,n_m} \delta_{v_i}$, where $(\alpha_{n_m},\mu_{1,n_m},\ldots,\mu_{n_m,n_m})$ is an element of the left-hand side of~(\ref{eq - existence of approximants}). Then the following statements hold.
\begin{enumerate}
\item $\lim_{m\rightarrow\infty}\int_{\Omega}e^{\alpha_{n_m}\sum_{i=1}^{n_{m}}\mu_{i,n_{m}}\,v_i}\,dQ=\inf\left\{\int_{\Omega}e^{\int_{\overline{\mathcal{V}}}v\,d\xi(v)}\,dQ:\xi\in\Upsilon\right\}$ and
\item Every accumulation point of $\left\{\xi_m\right\}_{m\geq1}$, in the weak-star topology of $C\left(\overline{\mathcal{V}}\right)^{*}$, is an element of \\ $\arginf\left\{\int_{\Omega}e^{\int_{\overline{\mathcal{V}}}v\,d\xi(v)}\,dQ:\xi\in\Upsilon\right\}$.
\end{enumerate}
\end{theorem}
\begin{proof}
We first prove the existence result~(\ref{eq - existence of approximants}). The proof proceeds by the direct method. We must consider two cases: where the partition is such that the optimal value~(\ref{eq - finite program}) is either infinite or finite. In the former case, existence holds trivially since it implies that every element in the domain is a solution. It is the latter case which requires further analysis. In that case, under the assumptions of the theorem, it is sufficient to establish the map $(\alpha,\mu_{1},\ldots,\mu_{n_m})\mapsto \int_{\Omega}e^{\alpha\sum_{i=1}^{n_m}\mu_i\,v_i}\,dQ$ is lower semi-continuous and coercive on its domain of definition, with respect to the Euclidean norm on $\mathbb{R}^{n_m+1}$. The steps for establishing these properties of the map are similar to their counterparts in the proofs of Theorems~\ref{Thm - Fenchel} and~\ref{Thm - existence reparametrization}, but with the appropriate adjustment based on using the Euclidean norm on $\mathbb{R}^{n_m+1}$ instead of the other norms used in the preceding theorems. We omit the details for brevity.

\par Now, we prove the two parts of the theorem by the direct method. \\
{\bf Part 1}. By Theorem~\ref{Thm - existence reparametrization} $\arginf\left\{\int_{\Omega}e^{\int_{\overline{\mathcal{V}}}v\,d\xi(v)}\,dQ:\xi\in\Upsilon\right\}\neq\emptyset$, and let $\xi_0=\alpha_0\,\mu_0$ be an element of that set. Given $m$,
\begin{align*}
\int_{\Omega}e^{\int_{\overline{\mathcal{V}}}v\,d\xi_0(v)}\,dQ\leq \int_{\Omega}e^{\int_{\overline{\mathcal{V}}}v\,d\xi_m(v)}\,dQ & =  \int_{\Omega}e^{\alpha_{n_m}\sum_{i=1}^{n_{m}}\mu_{i,n_m}\,v_i}\,dQ
 \leq\int_{\Omega}e^{\alpha_0\sum_{i=1}^{n_m}\mu_{0}(E_{i,m})\,v_i}\,dQ,
\end{align*}
holds, for \emph{any} choice of $v_i\in E_{i,m}$ for each $i$. The first inequality holds because $\xi_m=\alpha_{n_m}\sum_{i=1}^{n_m}\mu_{i,n_m} \delta_{v_i}\in\Upsilon$. The second inequality is due to $\alpha_{n_m},\mu_{1,n_m},\ldots,\mu_{n_m,n_m}$ being a solution of~(\ref{eq - finite program}), as the elements $\alpha_0,\mu_0(E_{1,n_m}),\ldots,\mu_0(E_{n_m,m})$ are in the domain of the finite program~(\ref{eq - finite program}). Hence, to establish the desired result, we must show 
$\lim_{m\rightarrow\infty}\int_{\Omega}e^{\alpha_0\sum_{i=1}^{n_m}\mu_{0}(E_{i,m})\,v_i}\,dQ=\int_{\Omega}e^{\int_{\overline{\mathcal{V}}}v\,d\xi_0(v)}\,dQ$.
Toward that end, for any $A>0$ let $\varphi\in C(\mathbb{R})$ be such that
$\begin{cases}
  \varphi(x)=1  & x\in[-A,A]  \\
  \varphi(x)=0  & x\in[-2A,2A]^c \\
  0\leq\varphi(x)\leq 1 & \text{ otherwise}
\end{cases}.$
We can bound the difference $\int_{\Omega}e^{\alpha_0\sum_{i=1}^{n_m}\mu_{0}(E_{i,m})\,v_i}\,dQ-\int_{\Omega}e^{\int_{\overline{\mathcal{V}}}v\,d\xi_0(v)}\,dQ$ from above by
\begin{align*}
\int_{\Omega}e^{\alpha_0\sum_{i=1}^{n_m}\mu_{0}(E_{i,m})\,v_i}\left(1-\varphi\left(\sum_{i=1}^{n_m}\mu_{0}(E_{i,m})\,v_i\right)\right)\,dQ+\int_{\Omega}e^{\int_{\overline{\mathcal{V}}}v\,d\xi_0(v)}\left(1-\varphi\left(\int_{\overline{\mathcal{V}}}v\,d\mu_0(v)\right)\right)\,dQ\\
+ \int_{\Omega}e^{\alpha_0\sum_{i=1}^{n_m}\mu_{0}(E_{i,m})\,v_i}\varphi\left(\sum_{i=1}^{n_m}\mu_{0}(E_{i,m})\,v_i\right)\,dQ-\int_{\Omega}e^{\int_{\overline{\mathcal{V}}}v\,d\xi_0(v)}\varphi\left(\int_{\overline{\mathcal{V}}}v\,d\mu_0(v)\right)\,dQ,
\end{align*}
which in turn, is bounded from above by
\begin{align}
\int_{\Omega}e^{\alpha_0\sum_{i=1}^{n_m}\mu_{0}(E_{i,m})\,v_i}1\left[\left|\sum_{i=1}^{n_m}\mu_{0}(E_{i,m})\,v_i\right|>A\right]\,dQ+\int_{\Omega}e^{\int_{\overline{\mathcal{V}}}v\,d\xi_0(v)}1\left[\left|\int_{\overline{\mathcal{V}}}v\,d\mu_0(v)\right|>A\right]\,dQ\label{eq - proof thm 3 000}\\
+ \int_{\Omega}e^{\alpha_0\sum_{i=1}^{n_m}\mu_{0}(E_{i,m})\,v_i}\varphi\left(\sum_{i=1}^{n_m}\mu_{0}(E_{i,m})\,v_i\right)\,dQ-\int_{\Omega}e^{\int_{\overline{\mathcal{V}}}v\,d\xi_0(v)}\varphi\left(\int_{\overline{\mathcal{V}}}v\,d\mu_0(v)\right)\,dQ.\label{eq - proof thm 3 001}
\end{align}

\par Observe that for each $A$, the term in~(\ref{eq - proof thm 3 001}) tends to zero as $m\rightarrow\infty$, because $x\mapsto e^{\alpha_0 x}\varphi(x)$ is a bounded and continuous function on $\mathbb{R}$ and  $\sum_{i=1}^{n_m}\mu_{0}(E_{i,m})\,v_i\stackrel{L_1(Q)}{\longrightarrow}\int_{\overline{\mathcal{V}}}v\,d\mu_0(v)$ (by Lemma~\ref{Lemma - Riemann Sum}). Focusing on the first term in~(\ref{eq - proof thm 3 000}), observe that by an application of the Cauchy-Schwartz inequality, it is bounded from above by $\left[Q\left(\left|\sum_{i=1}^{n_m}\mu_{0}(E_{i,m})\,v_i\right|>A\right)\,\int_{\Omega}e^{2\alpha_0\sum_{i=1}^{n_m}\mu_{0}(E_{i,m})\,v_i}\,dQ\right]^{1/2}$. Now this resulting term is bounded from above by $\left[A^{-1}\sup_{\gamma\in\Gamma}\int_{\Omega}|f_\gamma|\,dQ\,\sup_{\gamma\in\Gamma}\,\int_{\Omega}e^{-2\alpha_0f_\gamma}\,dQ\right]^{1/2}$, which follows from applying Markov's inequality to $Q\left(\left|\sum_{i=1}^{n_m}\mu_{0}(E_{i,m})\,v_i\right|>A\right)$ and using the convexity of the exponential function to obtain 
$$\int_{\Omega}e^{2\alpha_0\sum_{i=1}^{n_m}\mu_{0}(E_{i,m})\,v_i}\,dQ\leq \sum_{i=1}^{n_m}\mu_{0}(E_{i,m})\int_{\Omega}e^{2\alpha_0v_i}\,dQ\leq\sup_{\gamma\in\Gamma}\,\int_{\Omega}e^{-2\alpha_0f_\gamma}\,dQ,$$
 where we have used $\sum_{i=1}^{n_m}\mu_{0}(E_{i,m})=1$ and substituted out the $v_i$ in terms of the moment functions (see~(\ref{eq - moment fxns set})).
As $A$ is arbitrary, we have
$\lim_{A\rightarrow\infty}\left[A^{-1}\sup_{\gamma\in\Gamma}\int_{\Omega}|f_\gamma|\,dQ\,\sup_{\gamma\in\Gamma}\,\int_{\Omega}e^{-2\alpha_0f_\gamma}\,dQ\right]^{1/2}=0$ under Part (ii) of Assumption~\ref{Assumption - Fenchel} and Assumption~\ref{Assump - Approx}. Finally, note that the second term in~(\ref{eq - proof thm 3 000}) does not depend on $m$ and only on $A$. It also vanishes as $A\rightarrow\infty$ since $\int_{\Omega}e^{\int_{\overline{\mathcal{V}}}v\,d\xi_0(v)}\,dQ<\infty$. Therefore, by a sandwiching argument,
\begin{align}\label{eq - proof thm 3 4}
\lim_{m\rightarrow\infty}\int_{\Omega}e^{\alpha_{n_m}\sum_{i=1}^{n_{m}}\mu_{i,n_{m}}\,v_i}\,dQ=\int_{\Omega}e^{\int_{\overline{\mathcal{V}}}v\,d\xi_0(v)}\,dQ,
\end{align}
holds. This concludes the proof of the first part of the theorem.

\par {\bf Part 2}. Let $\xi^\star$ be an accumulation point of $\{\xi_m\}_{m\geq1}$ in the weak-star topology of $C\left(\overline{\mathcal{V}}\right)^*$. Therefore, there exists a subsequence $\{m_\ell\}_{\ell\geq1}$ such that
$\xi_{m_{\ell}}\stackrel{w^{*}}{\longrightarrow}\xi^{\star}$. By observing that
\begin{align*}
\int_{\Omega}e^{\int_{\overline{\mathcal{V}}}v\,d\xi_0(v)}\,dQ\leq \int_{\Omega}e^{\int_{\overline{\mathcal{V}}}v\,d\xi_{m_\ell}(v)}\,dQ & =  \int_{\Omega}e^{\alpha_{n_{m_\ell}}\sum_{i=1}^{n_{m_\ell}}\mu_{i,n_{m_\ell}}\,v_i}\,dQ \leq\int_{\Omega}e^{\alpha_0\sum_{i=1}^{n_{m_\ell}}\mu_{0}(E_{i,{m_\ell}})\,v_i}\,dQ,
\end{align*}
holds, Part 1 of this theorem implies $\lim_{\ell\rightarrow\infty}\int_{\Omega}e^{\int_{\overline{\mathcal{V}}}v\,d\xi_{m_{\ell}}}\,dQ=\int_{\Omega}e^{\int_{\overline{\mathcal{V}}}v\,d\xi_0(v)}\,dQ$. Now using the fact that the map $\xi\mapsto\int_{\overline{\mathcal{V}}}v\,d\xi$ is continuous when $C\left(\overline{\mathcal{V}}\right)^*$ is given the weak-star topology and $L_1(Q)$ has the weak topology, it follows that $\int_{\overline{\mathcal{V}}}v\,d\xi_{m_{\ell}}\stackrel{w}{\longrightarrow}\int_{\overline{\mathcal{V}}}v\,d\xi^{\star}$. Applying the Skorohod Representation Theorem, there exists a probability space $(\Omega^\prime,\mathcal{F}^\prime,Q^\prime)$ and real-valued measurable functions $\{z_\ell\}_{\ell\geq 1}$ and $z$ on this probability space, such that
\begin{enumerate}[(a)]
\item $\{z_\ell\}_{\ell\geq1}$ and $z$ have the same probability distributions as $\{e^{\int_{\overline{\mathcal{V}}}v\,d\xi_{m_\ell}}\}_{\ell\geq 1}$ and $e^{\int_{\overline{\mathcal{V}}}v\,d\xi^\star}$, respectively, and
\item $\lim_{\ell\rightarrow\infty}z_\ell(\omega^\prime)=z(\omega^\prime)$ a.s.$-Q^\prime$.
\end{enumerate}
One can set $\Omega^\prime=[0,1]$ with $\{z_\ell\}_{\ell\geq1}$ and $z$ being the quantile functions of $\{e^{\int_{\overline{\mathcal{V}}}v\,d\xi_{m_\ell}}\}_{\ell\geq 1}$ and $e^{\int_{\overline{\mathcal{V}}}v\,d\xi^\star}$, respectively. Consequently, $z_\ell(\omega^{\prime}),z(\omega^{\prime})\geq 0$ for all $\omega^{\prime}\in [0,1]$. Now we can use Fatou's Lemma to deduce the desired result:
\begin{align*}
\int_{\Omega}e^{\int_{\overline{\mathcal{V}}}v\,d\xi^\star}\,dQ=\int_{[0,1]}z\,dQ^{\prime}\leq\liminf_{\ell\rightarrow\infty}\int_{[0,1]}z_\ell\,dQ^{\prime}=\liminf_{\ell\rightarrow\infty}\int_{\Omega}e^{\int_{\overline{\mathcal{V}}}v\,d\xi_{m_\ell}}\,dQ=\int_{\Omega}e^{\int_{\overline{\mathcal{V}}}v\,d\xi_0(v)}\,dQ.
\end{align*} 
Therefore, $\xi^{\star}\in\arginf\left\{\int_{\Omega}e^{\int_{\overline{\mathcal{V}}}v\,d\xi(v)}\,dQ:\xi\in\Upsilon\right\}$ since $\xi_{0}$ is an element of that set. This concludes the proof.
\end{proof}

\par The first result of Theorem~\ref{Thm - Computation} is that the dual problem's optimal value can be approximated by the values of the finite-dimensional programs~(\ref{eq - finite program}). The second result of this theorem is that every weak-star accumulation point of a sequence of optimal solutions for the approximating programs~(\ref{eq - finite program}) is an optimal solution for the dual problem. Our use of the weak-star topology of $C\left(\overline{\mathcal{V}}\right)^{*}$ to characterize the approximation scheme's limiting behavior comes from the result of Theorem~\ref{Thm - existence reparametrization}, where we have used it to establish the weak-star lower semi-continuity and coercivity of the objective function~(\ref{eq - objective fxn reparametrized}) on $\Upsilon$.

\begin{remark}\label{Remark on Gamma Convergence}
The results of Theorem~\ref{Thm - Computation} are typical properties of Gamma-convergence; see, for example, Corollary 7.20 in~\citet{Dal_Maso} and Theorem~2.3 of \citet{Mena-Lerma-2005}. The proof of Theorem~\ref{Thm - Computation} does not use Gamma-convergence to establish the results. Furthermore, we establish Part 1 without imposing the weak-star convergence of $\left\{\xi_m\right\}_{m\geq1}$ in $\Upsilon$ --  the sequence of minimizers from our approximation scheme. Note that without this weak-star convergence, the only deduction from applying Gamma-convergence results in our setup, such as those mentioned above, would be $$\limsup_{m\rightarrow\infty}\int_{\Omega}e^{\alpha_{n_m}\sum_{i=1}^{n_{m}}\mu_{i,n_{m}}\,v_i}\,dQ=\int_{\Omega}e^{\int_{\overline{\mathcal{V}}}v\,d\xi_0(v)}\,dQ,$$ only if $\left\{\xi_m\right\}_{m\geq1}$ has at least one accumulation point. However, approximating the optimal value~(\ref{eq - Reparametrized Fenchel Dual Problem}) using the sequence of optimal values~(\ref{eq - finite program}) is feasible, regardless of whether $\left\{\xi_m\right\}_{m\geq1}$ has at least one accumulation point. The implication of this result for practice is that using our approximation scheme, it is not necessary to check that the sequence $\left\{\xi_m\right\}_{m\geq1}$ has an accumulation point a priori or even to pass through a subsequence in approximating the optimal value~(\ref{eq - Reparametrized Fenchel Dual Problem}).
\end{remark}

\par The most appealing feature of this approximation scheme is its simplicity, as one only needs to solve a finite-dimensional convex stochastic program in practice, and such programs can be solved numerically using off-the-shelf simulation-based computational routines. Prior to implementing this finite program, the practitioner must decide on a level of accuracy for the Riemann sum approximation in~(\ref{eq - Riemann Sum Consruction}); that is, a choice of $\epsilon$ and a corresponding partition $\{E_{i}\}_{i=1}^{n}$ of $\overline{\mathcal{V}}$, as well as the choice of $v_i\in E_i$ for each $i$. These are the ingredients/parameters for setting up the finite program.

\par For each $\epsilon>0$ and corresponding partition, there is more than one choice of $v_i\in E_i$ for each $i$ that delivers the accuracy~(\ref{eq - Riemann Sum Consruction}). However, this accuracy is uniform in the choice of $v_i\in E_i$ for each $i$. Furthermore, the result of Theorem~\ref{Thm - Computation} is that it holds for \emph{any} sequence of choices of those functions. Regarding implementation in practice, it is important to offer guidance on the choice of $v_i\in E_i$ for each $i$. The result of Theorem~\ref{Thm - Fenchel} suggests selecting $v_i\in \overline{B}\cap E_i$ for each $i$. The catch is that we do not know the set $\overline{B}$. However, $0=\int_{\Omega}v\,dP_Q=\frac{\int_{\Omega}ve^{y_0}\,dQ}{\int_{\Omega}e^{y_0}\,dQ}\geq \int_{\Omega}v\,dQ$ $\forall v\in\overline{B}$, by Lemma~\ref{Lemma - Inequality}, implying that $\overline{B}\subset \left\{v\in\mathcal{V}: \int_{\Omega}v\,dQ\leq0\right\}$. Consequently, we can select
 \begin{align}\label{eq - Selection of v_i}
 v_i\in \left\{v\in\mathcal{V}: \int_{\Omega}v\,dQ\leq0\right\}\cap E_i\quad \text{for each $i$},
 \end{align}
provided these intersections are non-empty. This approach can be advantageous in practice, since it can speed up numerical calculations when $\left\{v\in\mathcal{V}: \int_{\Omega}v\,dQ\leq0\right\}$ is a proper subset of $\mathcal{V}$. We demonstrate this point in Section~\ref{Section - Examples} using a numerical example based on the interval censored data problem described in Section~\ref{Section - OR Example}.

\section{Discussion}\label{Section - Discussion}
\par This section presents a discussion of the scope of our main results and implications for practice.
\subsection{Theorem~\ref{Thm - Fenchel}}
\par The results of Theorem~\ref{Thm - Fenchel} are not surprising as they mirror results one would expect in $I$-projection problems with a finite number of moment inequality restrictions (i.e., $\Gamma$ is finite). However, our results are new since the infinite $\Gamma$ setup hasn't been addressed in the $I$-projection literature. A related paper to ours is~\citet{Csiszar-conditional}, who developed results on the existence and exponential family representation of the \emph{generalized} $I$-projection, where the probability distributions' sample space is a locally convex topological vector space whose mean value is constrained to a convex set. By contrast, the results of Theorem~\ref{Thm - Fenchel} concern the $I$-projection, with the distinction between it and its generalized counterpart being their membership in the constraint set.
Notably, the generalized $I$-projection exists under weak conditions but isn't necessarily a member of the constraint set of PDs. While this property of the generalized $I$-projection is advantageous in some applications, such as statistical mechanics, it is not suitable in other applications, such as constrained statistical inference and operations research, where obeying the restrictions that define the constraint set is essential. $I$-projections, on the other hand, are elements of the constraint set but require additional regularity conditions. Furthermore, it should be noted that the $I$-projection and the generalized $I$-projection coincide when the former exists (see Remark 1 of~\citet{csiszar-matus}). Consequently, it is of interest to find conditions that yield the existence of $I$-projections in practice.

\par Assumption~\ref{Assumption - Fenchel} delivers the $I$-projections' existence when there can be infinitely many moment inequality restrictions defining the constraint set. This assumption enables the application of Theorem~4 of~\citet{Csiszar-conditional} in our setup to deduce that the constraint set $\mathcal{M}$ has the Sanov Property~\cite{Sanov}). Given the arbitrary nature of $\Gamma$ and the moment functions $\{f_{\gamma}:\gamma\in\Gamma\}$ in the derivation of our results, it is instructive to compare and contrast our results on existence and exponential family representation with the results of Theorem~4 in his paper.

\par Csisz\'{a}r focuses on the generalized $I$-projection of a reference probability measure $R$ that is defined on a locally convex topological vector space $\mathcal{T}$, where the constraint set imposes a convexity restriction on the expectation resultant of probability measures defined on $\mathcal{T}$. The expectation resultant of a probability measure $W$ defined on $\mathcal{T}$ is defined as the following Gelfand-Pettis integral:
\begin{align}\label{eq - csizar Weak Integral}
E[W]=t_0\quad\text{if}\quad\int_{\mathcal{T}} \Psi(t)\, dW(t)=\Psi(t_0)\quad\forall \Psi\in \mathcal{T}^*,
\end{align}
if such a $t_0$ exists, where else $E[W]$ is undefined. Here $\mathcal{T}^*$ is the topological dual of $\mathcal{T}$. The expectation resultant is a generalization of the notion of mean value of a measure on Euclidean spaces to topological vector spaces. The constraint set is $\left\{W: E[W]\in C\right\}$, where $C\subset \mathcal{T}$ is convex. To ensure existence of this weak integral so that his constraint set is well-defined, Csisz\'{a}r assumes that the reference probability measure $R$ is convex-tight. This property of $R$ means that there is an increasing sequence of compact and convex subsets of $\mathcal{T}$, $K_1\subset K_2\cdots,$ such that $R(K_n)\rightarrow 1$ as $n\rightarrow+\infty$. Its advantage is that it permits the use of the constraint set
$$\left\{W: E[W]\in C\quad\text{and}\quad W(K_n)=1\;\text{for some $n$}\right\}$$ in analyzing of the $I$-projection problem, ensuring the existence of the weak integral $E[W]$ in~(\ref{eq - csizar Weak Integral}). The convex-tight property of $R$ is also a crucial technical condition that facilitates the development of existence and exponential family representation results in Csiz\'{a}r's setup.

\par The $I$-projection problem of our paper can be formulated in terms of Csisz\'{a}r's general setup, but a general treatment is not possible, as the practical choice of $\mathcal{T}$ and $\sigma$-algebra, $\mathbb{T}$, can depend on $\Gamma$ and/or additional properties of the path functions, $\gamma\mapsto f_\gamma(\omega)$ for $\omega\in\Omega$. Let $\Phi:\Omega\rightarrow\mathcal{T}$ with $\Phi(\omega)$ is defined as the function $\gamma\mapsto f_{\gamma}(\omega)$ on the domain $\Gamma$, be measurable $\mathcal{F}/\mathbb{T}$. Accordingly, the constraint set, $\mathcal{M}$, in Csisz\'{a}r's setup then has the formulation 
\begin{align}\label{eq - Const Set Gen}
\left\{W\,\text{is a PD on}\,\left(\mathcal{T},\mathbb{T}\right): W=P\circ \Phi^{-1},\,\text{for}\,P\,\text{such that}\,\frac{dP}{dQ}\in\mathcal{M}\right\},
\end{align} 
where $W=P\circ \Phi^{-1}$ is the pushforward measure of $P$.

\par A general treatment is possible when $\text{card}(\Gamma)\leq\text{card}(\mathbb{N})$: the specification $\mathcal{T}=\mathbb{R}^{\Gamma}$, the vector space of all real-valued functions, however irregular, and $\mathbb{T}=\mathcal{R}^{\Gamma}$,
the product $\sigma$-algebra. In this setting, $\mathcal{R}^{\Gamma}$ coincides with the Borel $\sigma$-algebra generated by the product topology (e.g., Lemma 2.1 in~\citet{kallenberg2021foundations}). Consequently, this choice is suitable since $\mathbb{R}^{\Gamma}$ with the product topology becomes a separable $F$-space, where separability follows by an application of Theorem~2.46 in~\citet{Patty}. Note that being an $F$-space means the topology is induced by a complete invariant metric (Part (e) in Section 1.8 of~\citet{Rudin-Book}). Without loss of generality, suppose that $\Gamma=\mathbb{N}$, then by Theorem 2.50 in~\citet{Patty}, the metric is given by $$d(x,y)=\sup\left\{\min\{|x_i-y_i|,1\}/i,i\in\mathbb{N}\right\}\quad \text{for all}\, x,y\in\mathbb{R}^{\Gamma}.$$ Finally, $\mathbb{R}^{\Gamma}$ equipped with this topology is also locally convex, rendering it a separable Fr\'{e}chet space, and PDs defined on such spaces are necessarily convex-tight (see p.775 in \citet{Csiszar-conditional}). The significance of convex-tightness of these PDs is that we can define expectation as a Gelfand-Pettis integral, as in~(\ref{eq - csizar Weak Integral}), since such PDs concentrate on compact convex sets. 

\par Now consider the case $\text{card}(\Gamma)\geq\text{card}(\mathbb{R})$ (i.e., assuming the Continuum Hypothesis). While the specification $\left(\mathcal{T},\mathbb{T}\right)=\left(\mathbb{R}^{\Gamma},\mathcal{R}^{\Gamma}\right)$
is still feasible, $\mathcal{R}^{\Gamma}$ is inadequate because it can fail to contain subsets of $\mathbb{R}^{\Gamma}$ with desirable regularity properties, e.g., linear functions, continuous functions, polynomials, and constants (see the discussion in Section 36 in~\citet{Billingsley1}). For this reason, a more practical alternative approach is to consider the properties of the path functions $\left\{\gamma\mapsto f_\gamma(\omega), \omega\in\Omega\right\}$ in determining a suitable embedding. A leading scenario that covers many practical examples is when those properties enable specification of $\mathcal{T}$ as a Banach space with a suitable choice for $\mathbb{T}$, as PDs on such spaces are also necessarily convex-tight (see p.775 in \citet{Csiszar-conditional}). A canonical example is the setup where $\Gamma\subset\mathbb{R}$ is compact, and the path functions are continuous for each $\omega\in\Omega$. Then, the specification 
$\left(\mathcal{T},\mathbb{T}\right)=\left(C(\Gamma),\mathcal{B}\left(C(\Gamma)\right)\right)$ fits the bill, where $\mathcal{B}\left(C(\Gamma)\right)$ is the Borel $\sigma$-algebra generated by the norm topology.  

\par Concerning the examples on interval censored data (in Section~\ref{Section - Examples}) and stochastic dominance (in the Supplementary Material), these can also be formulated into Csisz\'{a}r's setup using this approach. In those examples, the moment functions are uniformly bounded on $\Omega\times\Gamma$. Therefore, the specification $\mathcal{T}=L_\infty\left(\Gamma,\mathcal{S},\nu\right)$, the Lebesgue space of bounded functions on a measure space $\left(\Gamma,\mathcal{S},\nu\right)$, is suitable, and equipping this function space with the essential supremum norm renders it a Banach space (e.g., Theorem 6.8(c) in \citet{Folland}). Accordingly, we can consider various $\sigma$-algebras on this function space, depending on the desired measurability, such as the Ball $\sigma$-algebra and the Borel $\sigma$-algebra (generated by the norm topology). Therefore, we can apply Theorem 4 of~\citet{Csiszar-conditional} in our setup to deduce the desired result.

\par Given a random sample, $X_1, X_2,X_3,\ldots, X_n,$ from the PD $Q$, denote by $Q^{\otimes_n}$ their joint PD. Then, under Assumption~\ref{Assumption - Fenchel}, the empirical distribution $\hat{Q}_n$ of the sample satisfies
\begin{align*}
\lim_{n\rightarrow\infty}n^{-1}\log Q^{\otimes_n}\left[\hat{Q}_n\in\mathcal{M}\right]=-m(p_Q),
\end{align*}
where $m(\cdot)$ is the objective function in~(\ref{eq - KL Min Problem}) and $p_Q$ is given by~(\ref{eq - Representation}). In words, the Sanov property describes the probability bound of the event $\{\hat{Q}_n\in\mathcal{M}\}$, showing that it is exponentially small with rate function given by the optimal value of the $I$-projection problem. Large deviation results of this sort have been successfully applied to studying the probability of rare events, in areas such as  finance~(e.g.,~\citet{BOYLE2007971}), optimal inference in econometrics and statistics (e.g.,~\citet{Hoeffding-1965} and~\citet{Canay}), data compression algorithms in information theory (e.g.,~\citet{Dembo-Kontoyiannis}), stochastic approximation methods in optimization~(e.g.,~\citet{book-SAA}), and thermodynamics in statistical physics (e.g.,~\citet{PhysRevE.102.052117}).

\subsection{Approximation of $I$-Projections}\label{Subsection - Approx I projection}
\par This section compares the approximation scheme of our paper with the one put forward by~\citet{bhattacharya2006} in the context of moment inequality constraints. Bhattacharya's algorithm applies to $I$-projection problems whose constraint sets can be represented as the finite intersection of sets that are variation-closed and convex, provided that closed-form expressions of the $I$-projections onto the intersecting sets can be obtained. The constraint sets our paper considers have this finite intersection representation, however, without imposing further structure on the $I$-projection problem in our setup, obtaining these individual $I$-projections can be a very challenging (if not impossible). By contrast, the advantage of our approximation scheme is that it does not rely upon the use of such closed-form expressions to be applicable in practice.

\par We discuss two examples of marginal stochastic order that intersect the setup in Section~4 of~\citet{bhattacharya2006}. These examples illustrate how closed-form expressions of $I$-projections can be obtained by imposing a number of $L_2(Q)$ restrictions on the $I$-projection problem. A typical restriction of this sort is $\left\|p\right\|_{L_2(Q)}<\infty$, where $p$ is a density function. It restricts the shape of $p$ by excluding densities from the constraint set that blow up too rapidly near some point. While convenient and simple, such shape restrictions can be confining in practice.

\par The first example considers the marginal stochastic order constraint set with a given marginal distribution -- see Example~\ref{Example - SD} below. As such a constraint set can be formulated as an isotonic cone restriction on $Q$-densities, results from~\citet{Bhatacharya-Dykstra} that connect infinite-dimensional $I$-projection problems with infinite-dimensional least-squares problems are useful for obtaining closed-form solutions of the former. However, this approach to solving the $I$-projection problem comes at the expense of imposing $L_2(Q)$ restrictions on the isotonic cone, the constraint set of $Q$-densities, and the $Q$-density of the given marginal distribution. This example is a moment inequality model, and we demonstrate the applicability of our approximation scheme without having to impose any of those $L_2(Q)$ restrictions.
\begin{example}\label{Example - SD}
Let $\Omega=[0,1]$, $Q$ be the uniform distribution on $\Omega$, and let $G$ be a given cumulative distribution function defined on $\Omega$. Consider the constraint set
\begin{align}\label{eq - constraint set Example 1}
\mathcal{M}=\left\{p\in L_{1}(Q):m(p)<\infty \;\text{and}\;\int_{0}^{\gamma}p\,dQ\leq G(\gamma)\;\;\forall\gamma\in[0,\overline{\gamma}] \right\},
\end{align}
where $0<\overline{\gamma}<1$ is given, so that $\Gamma=[0,\overline{\gamma}]$. In this example, the moment functions are $f_\gamma(\omega)=1[\omega\leq \gamma]-G(\gamma)$, and it is a straightforward task to show that
\begin{align*}
\int_{\Omega}e^{-\alpha f_\gamma}\,dQ   \leq e^{2\alpha} &<\infty\;\forall \alpha\geq 0\;\text{and}\;\forall \gamma\in\Gamma,\quad \sup_{\gamma\in\Gamma}\int_{\Omega}|f_{\gamma}|\,dQ\leq 2,\quad \text{and}\\
Q\left(\omega\in\Omega: f_{\gamma}(\omega)<0\;\forall \gamma\in\Gamma\right) & =Q\left(\omega\in(\overline{\gamma},1)\right)=1-\overline{\gamma}.
\end{align*}
Consequently, we invoke the result of Part 1 of Theorem~\ref{Thm - Fenchel} to solve this $I$-projection problem via its corresponding dual problem. This conclusion holds for any $\overline{\gamma}\in(0,1)$. From the above calculations, setting $\overline{\gamma}=1$ results in $Q\left(\omega\in\Omega: f_{\gamma}(\omega)<0\;\forall \gamma\in\Gamma\right)=0$, which violates Part (iii) of Assumption~\ref{Assumption - Fenchel}. Recall that in the proof of Theorem~\ref{Thm - Fenchel}, this condition was helpful for establishing the norm-coercivity of the objective function in the dual problem. Without this condition the map $y\mapsto\int_{\Omega}e^{y}$ may not be norm-coercive on $\mathcal{D}$. Proposition~S1.1 in the Supplementary Material establishes its coercivity provided the set of indices of moments that violate the inequality constraints under the reference PD, $Q$, given by $\Delta=\left\{\gamma\in\Gamma: G(\gamma)<\gamma\right\}$, satisfies $0<\inf\Delta$ and $\sup\Delta<1$, so that the closure of $\Delta$ is in the interior of $\Gamma$. Otherwise, we must use the theory of Gamma-Convergence to establish $\arginf\left\{\int_{\Omega}e^{y}\,dQ: y\in \mathcal{D}\right\}\neq\emptyset$ -- Proposition~S1.2 in the Supplementary Material presents this result. The structure we exploit in developing these results for those cases is that the moment functions satisfy $\sup_{\gamma\in\Gamma,\omega\in\Omega}|f_\gamma(\omega)|<\infty$. Then existence of the $I$-projection of $Q$ on the constraint set~(\ref{eq - constraint set Example 1}) and its exponential family representation for its $Q$-density follows from steps identical to those in the proof of Theorem~\ref{Thm - Fenchel}.
 
\par The least-squares approach to solving this $I$-projection problem uses results from Bhatacharya and Dykstra~\citet{Bhatacharya-Dykstra}, which apply only if $\mathcal{M}$ in~(\ref{eq - constraint set Example 1}) satisfies $\mathcal{M}\subset L_2(Q)$ and $\frac{dG}{dQ}\in L_2(Q)$. This additional structure on $\mathcal{M}$ and $G$ constrains the shapes of the feasible $Q$-densities and $G$, which can be restrictive in practice. An example of a distribution $G$ such that $\frac{dG}{dQ}$ exists but $\frac{dG}{dQ}\notin L_2(Q)$, is $\frac{dG}{dQ}(\omega)=\omega^{-1/2}/2$ for $\omega\in\Omega$ and zero otherwise. For this specification of $\frac{dG}{dQ}$, the projection $E_{Q}\left[\frac{dG}{dQ}\mid\mathcal{C} \right]$ is not uniquely defined. To see why, note that it is the solution of the following least-squares problem $\min_{h\in\mathcal{C}}\left\|\frac{dG}{dQ}-h\right\|^{2}_{L_2(Q)}=\min_{h\in\mathcal{C}}\left(\int_{\Omega}h^{2}\,dQ-2\int_{\Omega}\left(\frac{dG}{dQ}\,h\right)\,dQ\right)+\int_{\Omega}\left(\frac{dG}{dQ}\right)^2\,dQ$,
where the objective function equals $\infty$ for each $h\in\mathcal{C}$, since $\int_{\Omega}\left(\frac{dG}{dQ}\right)^2\,dQ=\frac{1}{4}\int_{0}^{1}\frac{1}{\omega}\,d\omega=\infty$. By contrast, the proposed approach does not require any additional structure on the constraint set~(\ref{eq - constraint set Example 1}) and $G$ to approximate the $I$-projection.

\par The proposed approximation scheme applies to the constraint set~(\ref{eq - constraint set Example 1}) and is valid under \emph{any} specification of the cumulative distribution function $G$ and $\overline{\gamma}\in(0,1)$. It is also straightforward to implement in practice. The set of moment functions in this example is given by $\mathcal{V}=\left\{G(\gamma)-1[\omega\leq \gamma]:\gamma\in\Gamma\right\}$, and lemma~\ref{lemma - Example SD 2} establishes that it satisfies Assumption~\ref{Assump - L_1 and total boundedness} and $\mathcal{V}=\overline{\mathcal{V}}$. Since $\mathcal{V}$ is also uniformly bounded, Assumption~\ref{Assump - Approx} trivially holds. Thus, by Theorem~\ref{Thm - Computation}, we can approximate the $I$-projection using the proposed approximation scheme. In doing so, we have to construct a suitable partition of $\mathcal{V}$ for given a level of accuracy as in~(\ref{eq - Riemann Sum Consruction}). For a given $\epsilon>0$, a finite collection of real numbers $0=\gamma_0<\gamma_1<\gamma_2<\cdots<\gamma_n=\overline{\gamma}$ can be found so that $\gamma_j-\gamma_{j-1}\leq \epsilon/2$ and $G(\gamma_j)- G(\gamma_{j-1})\leq \epsilon/2$, and for all $1\leq j\leq n$. This can always be done in such a way that $n\leq 4+ 2/\epsilon$. With such a construction, we have the following finite partition of $\mathcal{V}$ of size $n$ that fits into our framework: $\mathcal{V}=\cup_{j=1}^n E_{j}$, where $E_{1} =\left\{G(\gamma)-1[\omega\leq \gamma]:\gamma\in[\gamma_0,\gamma_1]\right\},E_{2}=\left\{G(\gamma)-1[\omega\leq \gamma]:\gamma\in(\gamma_1,\gamma_2]\right\},\ldots,$ and \\ $ E_{n}=\left\{G(\gamma)-1[\omega\leq \gamma]:\gamma\in(\gamma_{n-1},\gamma_n]\right\}$.
\end{example}

\par The $I$-projection problem of Example~\ref{Example - SD} extends in a natural way to include more than one marginal distribution $G$ and different reference distributions $Q$. Section~4 of \citet{bhattacharya2006} considers this extension for the case of multiple marginal distributions where he also illustrates the dual form of his iterative approximation scheme. The next example compares our method with Bhattacharya's algorithm using this extension of Example~\ref{Example - SD}. It should be noted that Bhattacharya also imposes the $L_2(Q)$ structure to exploit results from~\citet{Bhatacharya-Dykstra} for executing his algorithm. Like with Example~\ref{Example - SD}, we demonstrate the broader applicability of our approximation scheme for moment inequality models.

\begin{example}\label{Example - Bhatacharya}
Let $Q=Q_{X_1,X_2}$ be a fixed bivariate probability distribution on $\Omega=\mathbb{R}^{2}.$ In this setting, $\omega=(x_1,x_2)$, and let $G_1$ and $G_2$ be given cumulative distribution function defined on $\mathbb{R}$.
 Consider the constraint set
\begin{align}\label{eq - constraint set Example 2}
\mathcal{M}=\left\{p\in L_{1}(Q):m(p)<\infty, \;\text{and}\;\int_{\Omega}1[x_i\leq\gamma]p\,dQ\leq G_i(\gamma)\;\forall\gamma\in\mathbb{R},\,i=1,2 \right\}.
\end{align}
Consequently, $\mathcal{V}=\mathcal{V}_1\cup\mathcal{V}_2$, where $\mathcal{V}_1=\left\{\left(G_1(\gamma)-1[x_1\leq \gamma]\right)\,1[x_2\in\mathbb{R}]:\gamma\in\mathbb{R}\right\}$ and \\
$\mathcal{V}_2=\left\{\left(G_2(\gamma)-1[x_2\leq \gamma]\right)\,1[x_1\in\mathbb{R}]:\gamma\in\mathbb{R}\right\}$. Firstly, it should be noted that this $I$-projection problem with constraint set $\mathcal{M}$ in~(\ref{eq - constraint set Example 2}) is \emph{not} covered by our setup. However, we can still establish existence of the $I$-projection. Then we discuss approximation of the $I$-projection using Theorem~\ref{Thm - Computation}, and contrast it with the dual form of the iterative algorithm put forward by~\citet{bhattacharya2006}.

\par It is a straightforward task to show that Parts (i) and (ii) of Assumption~\ref{Assumption - Fenchel}, hold, using steps similar to those in Example~\ref{Example - SD}. For brevity, we omit those details. Part (iii) of that assumption, like with $\overline{\gamma}=1$ in Example~\ref{Example - SD}, does not hold. However, one can establish results similar to those of Propositions S1.1 and S1.2 to show $\arginf\left\{\int_{\Omega}e^{y}\,dQ: y\in \mathcal{D}\right\}\neq\emptyset$, using steps identical to those in their proofs, depending on whether or not $\Delta=\left\{\gamma\in\mathbb{R}: \exists j\in\{1,2\},\, G_j(\gamma)<Q_{X_j}(-\infty,\gamma]\right\}$ is bounded. This approach is feasible since the moment functions satisfy $\sup_{\gamma\in\Gamma,\omega\in\Omega}|f_\gamma(\omega)|<\infty$, and we omit the details for brevity. Then existence of the $I$-projection of $Q$ on the constraint set~(\ref{eq - constraint set Example 2}) and its exponential family representation for its $Q$-density follows from steps identical to those in the proof of Theorem~\ref{Thm - Fenchel}. Lemma~\ref{lemma - Example Bhattacharya} establishes that $\mathcal{V}$ satisfies Assumption~\ref{Assump - L_1 and total boundedness} and that $\mathcal{V}=\overline{\mathcal{V}}$, holds. Since $\mathcal{V}$ is also uniformly bounded, Assumption~\ref{Assump - Approx} trivially holds. 

\par Thus, we can approximate the solution of this $I$-projection by constructing a suitable partition of $\mathcal{V}$. Given $\epsilon>0$, a finite collection of real numbers $-\infty=\gamma_0<\gamma_1<\gamma_2<\cdots<\gamma_n=\infty$ can be found so that $Q_{X_i}\left((\gamma_{j-1},\gamma_j)\right)\leq \epsilon/4$ and $G_i(\gamma_j)- G_i(\gamma_{j-1})\leq \epsilon/4$, for $i=1,2$ and for all $1\leq j\leq n$. This can always be done in such a way that $n\leq 6+ 4/\epsilon$. With such a construction, we have the following finite partition of $\mathcal{V}$ of size $2n$: $\mathcal{V}=\left(\cup_{j=1}^n E_{1,j}\right)\cup \left(\cup_{j=1}^n E_{2,j}\right)$, where 
\begin{align*}
E_{1,j} & =\left\{\left(G_1(\gamma)-1[x_1\leq \gamma]\right)\,1[x_2\in\mathbb{R}]:\gamma\in(\gamma_{j-1},\gamma_j]\right\}\quad\text{and}\\
E_{2,j} & =\left\{\left(G_2(\gamma)-1[x_2\leq \gamma]\right)\,1[x_1\in\mathbb{R}]:\gamma\in(\gamma_{j-1},\gamma_j]\right\}
\end{align*}
for $j=1,2,\ldots,n$. Therefore, one can approximate the $I$-projection using the proposed approximation scheme with the above partition of $\mathcal{V}$.

\par As the constraint set~(\ref{eq - constraint set Example 2}) has the representation $\mathcal{M}=\cap_{i=1}^{2}\mathcal{M}_i$, where
\begin{align*}
\mathcal{M}_i=\left\{p\in L_{1}(Q):m(p)<\infty, \;\text{and}\;\int_{\Omega}1[x_i\leq\gamma]p\,dQ\leq G_i(\gamma)\;\forall\gamma\in\mathbb{R}\right\},
\end{align*}
for $i=1,2$, Bhattacharya proposes to solve for the $I$-projection using the following dual problem:
\begin{align*}
\inf &\left\{\int_{\Omega}e^y\,dQ: y\in \mathcal{S}_1\oplus\mathcal{S}_2\right\}\;\;\text{where}\;\; y=y(u_1,u_2)=y_1(u_1)+y_2(u_2)\in\mathcal{S}_1\oplus\mathcal{S}_2\;\;\text{and}\\
\mathcal{S}_1 & =\left\{y(u_1,u_2):y(u_1,u_2)=y_1(u_1),\,y_1(u_1)\,\text{nondecreasing},\,\int_{\mathbb{R}} y_1(u_1)\,dG_1(u_1)=0\right\}, \\
\mathcal{S}_2 & =\left\{y(u_1,u_2):y(u_1,u_2)=y_2(u_2),\,y_2(u_2)\,\text{nondecreasing},\,\int_{\mathbb{R}} y_2(u_2)\,dG_2(u_2)=0\right\}.
\end{align*}
His cyclical descent algorithm for approximating the solution of this dual problem successively minimizes over each $y_i$ at a time while the remaining is held fixed. In particular, it follows these steps:
\begin{itemize}
\item Initialization: Set $y_{0,i}=0$ and begin with $n=1, i=1$.
\item Implementation:
\begin{enumerate}
\item Let $y_{n,1}$ denote the solution to $\inf\left\{\int_{\Omega}e^{y+y_{n-1,2}}\,dQ: y\in \mathcal{S}_1\right\}$.
\item Let $y_{n,2}$ denote the solution to $\inf\left\{\int_{\Omega}e^{y_{n,1}+y}\,dQ: y\in\mathcal{S}_2\right\}$.
\item Increase $n$ by 1 and repeat the two previous steps.
\end{enumerate}
\end{itemize}
To implement this algorithm, the practitioner must be able to obtain closed-form expressions for $y_{n,1}$ and $y_{n,2}$ at each iteration of the algorithm. Thus, to satisfy this requirement, Bhattacharya imposes $\mathcal{M}_1,\mathcal{M}_2\subset L_2(Q)$ and $\frac{dG_1}{dQ_{X_1}},\frac{dG_2}{dQ_{X_2}}\in L_2(Q)$. Like in Example~\ref{Example - SD}, this additional structure creates a way forward through the use of Theorem~3.3 in~\citet{Bhatacharya-Dykstra} to obtain closed-form expressions for the dual variables at each iteration. By contrast, our approach does not impose any additional structure to approximate this $I$-projection, and hence, is more widely applicable.
\end{example}

\subsection{Verification of Assumptions}
\par The marriage of Assumptions~\ref{Assumption - Fenchel} and~\ref{Assump - L_1 and total boundedness} are at the heart of our results. In practice, the PD $Q$ is given, so that verification of Assumption~\ref{Assumption - Fenchel} can be easily carried out by direct calculation, as demonstrated in Examples~\ref{Example - SD} and~\ref{Example - Bhatacharya}. Additionally, the set $\mathcal{V}$ is also given in practice, and we have verified that it satisfies Assumption~\ref{Assump - L_1 and total boundedness} in those examples by directly calculating the $L_1(Q)$ bracketing number of $\mathcal{V}$. The calculations of other measures of complexity can also be used to verify Assumption~\ref{Assump - L_1 and total boundedness} in practice, such as the $L_1(Q)$ covering, or packing numbers. Such calculations are heavily used in empirical process theory to characterize classes of (moment) functions that are Glivenko-Cantelli and/or Donsker, but in the $L_{r}(Q)$ norm for $r\geq1$. By the result of Lemma~\ref{Lemma - precompactness}, such sets of moment functions that are precompact in $L_{r}(Q)$ with $r>1$, are necessarily precompact in $L_1(Q)$. Thus, the results also cover precompact subsets of $L_r(Q)$ for any $r>1$. See Section~2.7 of~\citet{VDV-W} for some examples of precompact classes in $L_r(Q)$ spaces.

\par In applications, the $L_1(Q)$-closure of $\mathcal{V}$ (i.e., $\overline{\mathcal{V}}$) must be derived to implement the proposed approximation scheme. This is not a difficult task, as demonstrated by the calculations in Examples~\ref{Example - SD} and~\ref{Example - Bhatacharya}. The technique used in Lemmas~\ref{lemma - Example SD 2} and~\ref{lemma - Example Bhattacharya} generalizes to other classes of moment functions. The main tool for deriving $\overline{\mathcal{V}}$ from $\mathcal{V}$ is as follows. By considering a $L_1(Q)$ limit point of $\mathcal{V}$, given by $y$, there is a sequence $\{y_n\}_{n\geq1}\subset \mathcal{V}$ such that $y_n\stackrel{L_1(Q)}{\longrightarrow} y$, where the objective is to derive the form of $y$. As $y_n \stackrel{L_1(Q)}{\longrightarrow} y\implies y_n\stackrel{Q}{\longrightarrow} y$, there exists $\{n_i\}_{i\geq1}\subset \{n\}$ non-random and such that $\lim_{i\rightarrow\infty} y_{n_i}(\omega)=y(\omega)\;\text{a.s.-}Q$. This a.s.$-Q$ convergence can be used to deduce the form of the limit point $y$ based on the functional forms of elements in $\mathcal{V}$. In the examples described in the Supplementary Material and the next section, we find that the limit point satisfies $y\in\mathcal{V}$, implying that $\mathcal{V}=\overline{\mathcal{V}}$, holds. Once one obtains $\overline{\mathcal{V}}$, constructing a finite partition $\{E_i\}_{i=1}^{n}$ of it for a given $\epsilon>0$, such that~(\ref{eq - Riemann Sum Consruction}) holds, is also a simple task. The electronic companion illustrates the verification of Assumption~\ref{Assump - L_1 and total boundedness}, the construction of $\overline{\mathcal{V}}$ from $\mathcal{V}$, and the construction of the finite partition $\{E_i\}_{i=1}^{n}$ of $\overline{\mathcal{V}}$, in the context of unconditional and conditional stochastic dominance constraints.

\section{Random Interval Example}\label{Section - Examples}
\par This section illustrates the verification of Assumption~\ref{Assump - L_1 and total boundedness} and the construction of the partition of $\overline{\mathcal{V}}$ using the example on censored data in process monitoring discussed in Section~\ref{Section - OR Example}. Furthermore, we report the results of numerical experiments based on this example. In the Supplementary Material, we provide further examples based on first-order conditional and unconditional stochastic dominance constraints. It is challenging (if not impossible) to obtain a closed-form expression of the objective function in~(\ref{eq - finite program}) for a general reference PD $Q$ and set of moment functions $\{f_\gamma :\gamma\in\Gamma\}$. Whence, we solve the finite programs~(\ref{eq - finite program}) using the \emph{Sample Average Approximation} (SAA); see~\citet{book-SAA} for a survey of this method. The SAA principle is very general, and solves optimization problems tractably using Monte Carlo. In our setup, the objective function in~(\ref{eq - finite program}) is an expected value of a known function with respect to the PD $Q$. Now, since we can draw random samples from $Q$, the SAA approach entails replacing that objective function with its sample-analogue version based on a random sample from $Q$, and then to minimize it instead of the original objective function. Notably, it is not computationally costly to calculate this sample-analogue version of the objective function. Naturally, the quality of the SAA solution (e.g., in terms of convergence rates) would depend on specifications of $Q$ and $\{f_\gamma :\gamma\in\Gamma\}$, and the size of the random sample, among other parameters. A formal analysis of convergence and statistical validation of using SAA within our approximation scheme goes beyond the intended scope of the paper, and we leave it for future research. In this manuscript we only illustrate its use within our approximation scheme.
\par We have used MATLAB to implement the numerical experiments. The code employs the parallel computing capabilities of fmincon based on the sequential quadratic programming algorithm and includes gradient evaluation in the objective function for faster or more reliable computations. For high levels of $\epsilon$, we have implemented the approximation on a desktop computer with 10 cores and 30 gigabytes of RAM. For low levels of $\epsilon$, we have implemented the approximation scheme on a cluster using 12 cores with 100 gigabytes of RAM.

\subsection{$I$-Projection Problem}\label{SubSubsection - Random Set example}
\par Consider the ``population'' version of the $I$-projection problem~(\ref{eq - KL Min Problem-1}), given by
\begin{equation}\label{eq - KL Min Problem-2}
\begin{aligned}
\text{minimize}& \quad K(p,p^*)=\begin{cases} \int_{\Omega}p\log(p)\,dP^* &\mbox{if } p\geq0,\quad \int_{\Omega}p\,dP^*=1\\
\infty &\mbox{elsewhere},\end{cases} \\
\text{subject to}&\quad\int_{\Omega}\left(1\left[\omega\in C\right]-\sigma\left(C\right)\right)\, p\,dP^*\leq 0\;\forall C\in \mathcal{C}(\Omega),\quad p\in L_1(P^*),
\end{aligned}
\end{equation}
where have replaced $\hat{\sigma}$ in~(\ref{eq - KL Min Problem-1}) by its population counterpart, $\sigma$. In terms of the notation of Section~\ref{Section- Main Results}, the reference distribution is $Q=P^*$, the objective function is $m(p)=K(p,p^*)$, the moment functions are $f_{\gamma}(\omega)=1\left[\omega\in C\right]-\sigma\left(C\right),\, C\in \mathcal{C}(\Omega)$ with $\gamma=C$ and $\Gamma=\mathcal{C}(\Omega)$. Furthermore,
\begin{align}\label{eq - Example Random set moment function}
\mathcal{V}=\left\{\sigma\left(C\right)-1\left[\omega\in C\right]:C\in \mathcal{C}(\Omega)\right\}.
\end{align}

\par For simplicity, we specify $\Omega=[0,1]$, and the endpoints of the random interval $Y=[\underline{Y},\overline{Y}]$, given by the vector $(\underline{Y},\overline{Y})$, are independent with marginal distributions $\underline{Y}\sim U[0,1/2]$ and $\overline{Y}\sim U[1/2,1]$. The capacity functional of $Y$ under this specification is given by
\begin{align}\label{eq - Capacity fun Exp}
\sigma\left([a,b]\right)= \begin{cases}2a+\min\{2(b-a),1-2a\} &\mbox{if } a<1/2\\
2a &\mbox{if } a\geq1/2,\end{cases}
\end{align}
for every $[a,b]\subset \Omega$.

\par The $I$-projection problem~(\ref{eq - KL Min Problem-2}) closely resembles the one in Example~\ref{Example - SD} of Section~\ref{Subsection - Approx I projection}, where the main difference is the nature of the index set $\Gamma$. Another difference is that the CDF $G$ in Example~\ref{Example - SD} has been replaced by the capacity functional $\sigma$ in~(\ref{eq - Capacity fun Exp}). Like in Example~\ref{Example - SD}, this example satisfies all the conditions of Assumption~\ref{Assumption - Fenchel} except for Condition (iii), as $Q\left(\omega\in\Omega: f_{\gamma}(\omega)<0\;\forall \gamma\in\Gamma\right)=0$; e.g., consider the choice $C=[0,\epsilon]$ with $\epsilon<1/2$, which results in $\left\{\omega\in\Omega: f_{C}(\omega)<0\right\}=\emptyset$. Despite this violation, we can establish $\arginf\left\{\int_{\Omega}e^{y}\,dQ: y\in \mathcal{D}\right\}\neq\emptyset$, holds. The approach depends on whether or not 
\begin{align}\label{eq - Conditions Delta RS}
0<\inf\left\{\inf C:C\in\Delta\right\}\quad\text{and}\quad\sup\left\{\sup C:C\in\Delta\right\}<1,
\end{align}
hold, where $\Delta=\left\{C\in \mathcal{C}(\Omega): \sigma\left(C\right)<Q\left(C\right)\right\}$. If the conditions in~(\ref{eq - Conditions Delta RS}), hold, then the sets in $\Delta$ are bounded away from the endpoints of the interval $[0,1]$. In consequence, we can obtain norm-coercivity of the map $y\mapsto\int_{\Omega}e^{y}$ on $\mathcal{D}$ in this example using steps identical to those in the proof of Proposition~\ref{Prop - Example 1 Delta Interior}. Now if the conditions in~(\ref{eq - Conditions Delta RS}) do not hold, then the desired coercivity no longer holds. Hence, to move forward we can consider a sequence of problems
\begin{align*}
\inf\left\{\int_{\Omega}e^{y}\,dQ:y\in\mathcal{D}_n\right\}\quad\text{with}\quad\mathcal{D}_n=\left\{y\in\mathcal{D}:\alpha\leq \bar{\alpha}_n\right\}
\end{align*} 
where $\bar{\alpha}_n\nearrow\infty$ as $n\rightarrow\infty$, and establish $\arginf\left\{\int_{\Omega}e^{y}\,dQ: y\in \mathcal{D}\right\}\neq\emptyset$ using the theory of Gamma-Convergence. Because the moment functions satisfy $\sup_{\gamma\in\Gamma,\omega\in\Omega}|f_\gamma(\omega)|<\infty$, by following steps identical to those in the proofs of Proposition~\ref{Prop - Existence Example 1} and Lemmas~\ref{Lemma - Example 1} -~\ref{Lemma - Existence Primal problem} in the Supplementary Material, we can obtain the desired result. We omit these details for brevity. Finally, the existence of the minimizer in the $I$-projection problem~(\ref{eq - KL Min Problem-2}) and its exponential family representation for its $Q$-density follows from steps identical to those in the proof of Theorem~\ref{Thm - Fenchel}. Assumption~\ref{Assump - Approx} trivially holds since the moment functions are uniformly bounded under the aforementioned specifications. For brevity, we omit these technical details.

\par We note that the capacity functional $\sigma$ in~(\ref{eq - Capacity fun Exp}) is Lipschitz continuous:
\begin{align}\label{eq - V RS 0}
\left|\sigma\left([a_1,b_1]\right)-\sigma\left([a_2,b_2]\right)\right|\leq 2 \,d_{H}\left([a_1,b_1],[a_2,b_2]\right)\quad \forall [a_1,b_1],[a_2,b_2]\in\mathcal{C}(\Omega),
\end{align}
where $d_{H}\left([a_1,b_1],[a_2,b_2]\right)=\max\{|a_1-a_2|,|b_1-b_2|\}$ is the Hausdorff distance between the intervals $[a_1,b_1]$ and $[a_2,b_2]$, which follows from the form of $\sigma$. Furthermore, by the triangular inequality we must have that
\begin{align}\label{eq - V RS 1}
\int_{\Omega}|f_{C}-f_{C^\prime}|\,dQ \leq \left|\sigma(C)-\sigma(C^\prime)\right|+Q\left(C\ominus C^\prime\right),
\end{align}
holds for any $C$ and $C^\prime$ in $\mathcal{C}(\Omega)$, where $\ominus$ denotes the symmetric difference operation on sets. These inequalities are useful for establishing the following result on the set $\mathcal{V}$ in~(\ref{eq - Example Random set moment function}).
\begin{proposition}\label{Prop - V Random Set}
Let $\mathcal{V}$ and $\sigma$ be given by~(\ref{eq - Example Random set moment function}) and (\ref{eq - Capacity fun Exp}), respectively. Then $\mathcal{V}=\overline{\mathcal{V}}$, holds, and $\mathcal{V}$ satisfies Assumption~\ref{Assump - L_1 and total boundedness}.
\end{proposition}
\begin{proof}
See Appendix~\ref{Appendix - proof RS}. 
\end{proof}
\noindent Thus, we can approximate the solution of the $I$-projection using its dual via Theorem~\ref{Thm - Computation} by constructing a suitable partition of $\mathcal{V}$. The two inequalities~(\ref{eq - V RS 0}) and~(\ref{eq - V RS 1}) present a simple avenue forward in the construction. Given $\epsilon>0$, we can obtain a partition of the unit interval $[0,1]$ given by $0=d_0<d_1<d_2<\cdots<d_m=1$ such that $d_{i}-d_{i-1},Q([d_{i-1},d_{i}])<\epsilon/4$ for $i=0,1,\ldots,m$. Now let
$D_1=[0,d_1]$, and $D_i=(d_{i-1},d_i]$ $\forall i=2,3\ldots m$. Consider the collection $\mathcal{E}$ consisting of the following sets:
\begin{align*}
E^{1}_1 &=\left\{C\in \mathcal{C}(\Omega):C\subset D_1\right\}, E^{1}_i=\left\{C\in \mathcal{C}(\Omega):C\subset D_i\right\}\;\forall i=2,3\ldots m;\\
E^{2}_1 &=\left\{C\in \mathcal{C}(\Omega):C\cap D_1\neq \emptyset,C\cap D_2\neq \emptyset\right\},\,\text{and}\\
E^{2}_i &=\left\{C\in \mathcal{C}(\Omega):C\cap D_i\neq \emptyset,C\cap D_{i+1}\neq \emptyset\right\}\;\forall i=2,3\ldots m-1;\\
E^{3}_1 &=\left\{C\in \mathcal{C}(\Omega):C\cap D_1\neq \emptyset,C\cap D_2\neq \emptyset,C\cap D_3\neq \emptyset\right\},\,\text{and}\\
E^{3}_i &=\left\{C\in \mathcal{C}(\Omega):C\cap D_i\neq \emptyset,C\cap D_{i+1}\neq \emptyset,C\cap D_{i+2}\neq \emptyset\right\}\;\forall i=2,3\ldots m-2;\\
\vdots \\
E^{m}_1 &=\left\{C\in \mathcal{C}(\Omega):C\subset C\cap D_{i}\neq \emptyset\; \forall i=1,2,\ldots m \right\}.
\end{align*}
These sets form a finite partition of $\mathcal{C}(\Omega)$ of size $n=(m^2-m)/2+1$, and for each $E\in \mathcal{E}$
\begin{align*}
\int_{\Omega}|f_{C}-f_{C^\prime}|\,dQ \leq \epsilon\quad \forall C,C^\prime\in E,
\end{align*}
holds. Thus, the collection $\left\{\left\{-f_{\gamma}\in\mathcal{V}:\gamma\in E\right\},\;E\in\mathcal{E}\right\}$ forms a finite partition of $\mathcal{V}$ to which Theorem~\ref{Thm - Computation} applies.

\subsection{Numerical Experiments}
\par Now we illustrate the proposed approximation scheme numerically with the above example for values $\epsilon\in\{1,0.5,0.25,0.125,0.0833\}$ and $Q$ having a Kumaraswamy distribution. Recall that this distribution has CDF $1-(1-x^{a})^{b}$ for $x\in[0,1]$, where $a,b>0$. In the numerical example we set $a=1$ and $b=3$, yielding $Q\not\in\mathcal{M}$. This non-inclusion arises because $\sigma\left(C\right)< Q(C)$ for all $C\in C(\Omega)$ such that $C=[0,\ell]$ and $\ell<\frac{3-\sqrt{5}}{2}\approx 0.382$. Consequently, the $I$-projection yields a $Q$-density closest to this Kumaraswamy distribution, in the Kullback-Leibler sense, that is selectionable with respect to the capacity functional $\sigma$ of the random set $Y$. Since we do not have a closed-form expression for the objective function in~(\ref{eq - finite program}), we use the SAA to approximate it with $10^4$ draws from $Q$.

\begin{table}[pt]
\centering
\bigskip
\resizebox{10cm}{!}{
\begin{tabular}{c*{7}{c}}
\toprule
\multicolumn{1}{c}{$\epsilon$} &  \multicolumn{2}{c}{$n$}  & \multicolumn{2}{c}{Optimal Value~(\ref{eq - finite program})} & \multicolumn{2}{c}{Time}   \\
\cmidrule(lr){1-1}\cmidrule(lr){2-3}\cmidrule(lr){4-5}\cmidrule(lr){6-7} \\
 & $a$ & $b$ & $a$ & $b$ & $a$ & $b$\\
 \cmidrule(lr){2-3}\cmidrule(lr){4-5}\cmidrule(lr){6-7} \\
1 & 66 & 10 & 0.9871 & 0.9871 & 0.14s & 0.025s \\

0.5 & 276 &  36 & 0.9916 & 0.9916 & 1.3s & 0.152s  \\

0.25 & 1128 & 153 & 0.9897 & 0.9897 & 22.9s & 0.868s  \\

0.125 & $4560^*$ &  666 & $0.9886^*$ & 0.9886 & 86m$^*$ & 9.99s  \\

0.0833 & $10296^*$ & 1485 & $0.9853^*$ & 0.9853 & 27.3h$^*$ & 62.7s  \\

0.0625* & - & 2628 & - & 0.9842 & - & 486s \\

0.0313* & - & 10731 & - & 0.9826 & - & 35.6h \\
\hline
\end{tabular}
}
\caption{Output of approximation schemes based on $\mathcal{C}([0,1])$ and $\mathcal{C}\left(\left[0,\frac{3-\sqrt{5}}{2}\right]\right)$ in columns $a$ and $b$, respectively. Notation: "*" denotes implementation on a cluster.}\label{tab03}
\end{table}

\par For each $\epsilon$, we specified an equidistant grid $0=d_0<d_1<d_2<\cdots<d_m=1$ such that $d_{i}-d_{i-1},Q([d_{i-1},d_{i}])<\epsilon/4$ for $i=0,1,\ldots,m$. Letting $c_i=(d_{i}+d_{i-1})/2$ for $i=1,\ldots,m$, we implemented the approximation scheme using the following intervals:
\begin{align*}
D_i & \in E^1_i\quad \forall i=1\ldots,m;\, [c_i,c_{i+1}] \in E^2_i\quad \forall i=1\ldots,m-1;\\
[c_i,c_{i+2}] & \in E^3_i\quad \forall i=1\ldots,m-2;\ldots;\,\text{and}\;[c_1,c_{m}] \in E^{m}_1.
\end{align*}

\begin{figure}[pt]
    \centering
    \resizebox{15cm}{!}{
    \includegraphics[width=1\textwidth]{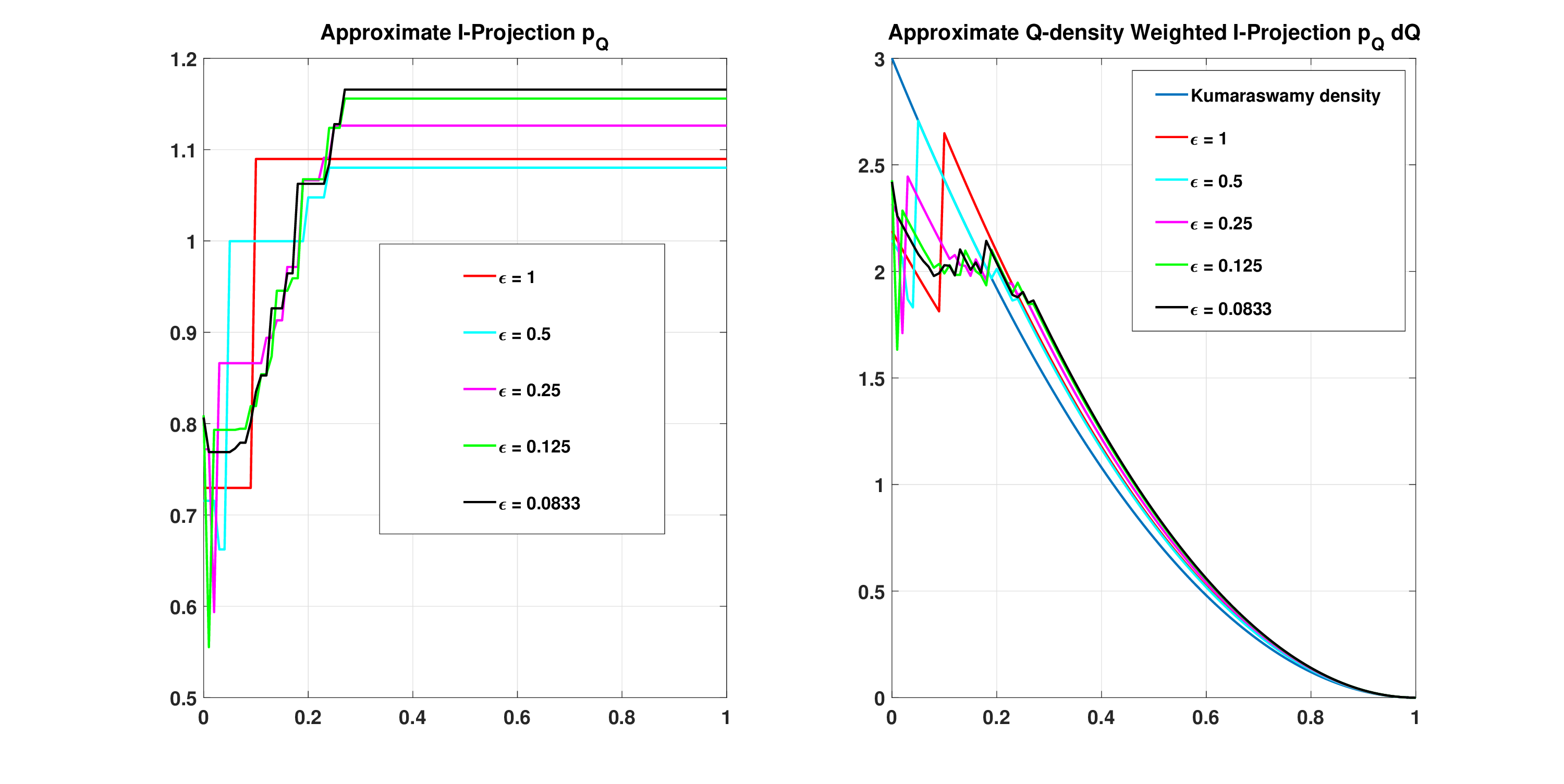}
}
    \caption{Graphs of the approximate $I$-projection, $p_Q$, and its $Q$-density weighted version, $p_Q\,dQ$.}
    \label{fig: Random Set}
\end{figure}

\noindent The $a$ columns in Table~\ref{tab03} reports the values of $\epsilon$ we have considered and their impact on $n$, the optimal minimal value of the approximating finite programs~(\ref{eq - finite program}), and the execution times of the numerical procedure. Figure~\ref{fig: Random Set} reports the approximate $I$-projections and their $Q$-density weighted versions.

\par The rapid growth in $n$ for lower levels of $\epsilon$ naturally leads to computational challenges by increasing the number of choice variables in the approximating finite program. Notice in the $a$ columns of Table~\ref{tab03} with $\epsilon = 0.0833$, the finite program had 10297 choice variables and required 27.3 hours to complete. To mitigate this computational challenge, the result in Part 2 of Theorem~\ref{Thm - Fenchel} can serve as a guide towards that end. Recall that this result establishes the moment functions corresponding to the binding inequalities, i.e. the set $\overline{B}$, are the only ones that enter the $I$-projection's representation. Consequently, having information on the set $\overline{B}$ can help reduce the computational burden with lower levels of $\epsilon$. The result of Lemma~\ref{Lemma - Inequality} implies $\overline{B}\subset\{v\in\mathcal{V}: \int_{\Omega}v\,dQ\geq0\}$, and one implements this information in the approximation scheme by constructing a suitable partition of $\{v\in\mathcal{V}: \int_{\Omega}v\,dQ\geq0\}$ instead of $\mathcal{V}$. We demonstrate this point by repeating the above numerical experiment with the identical setup except that now for each $\epsilon$, we have specified an equidistant grid $0=d_0<d_1<d_2<\cdots<d_m=\frac{3-\sqrt{5}}{2}$ as the interval $[0,(3-\sqrt{5})/2]$ because various subsets of it in $C(\Omega)$ satisfy $\sigma(C)\leq Q(C)$. The output correspond to the $b$ columns in Table~\ref{tab03}.

\begin{figure}[pt]
    \centering
    \resizebox{15cm}{!}{
    \includegraphics[width=1\textwidth]{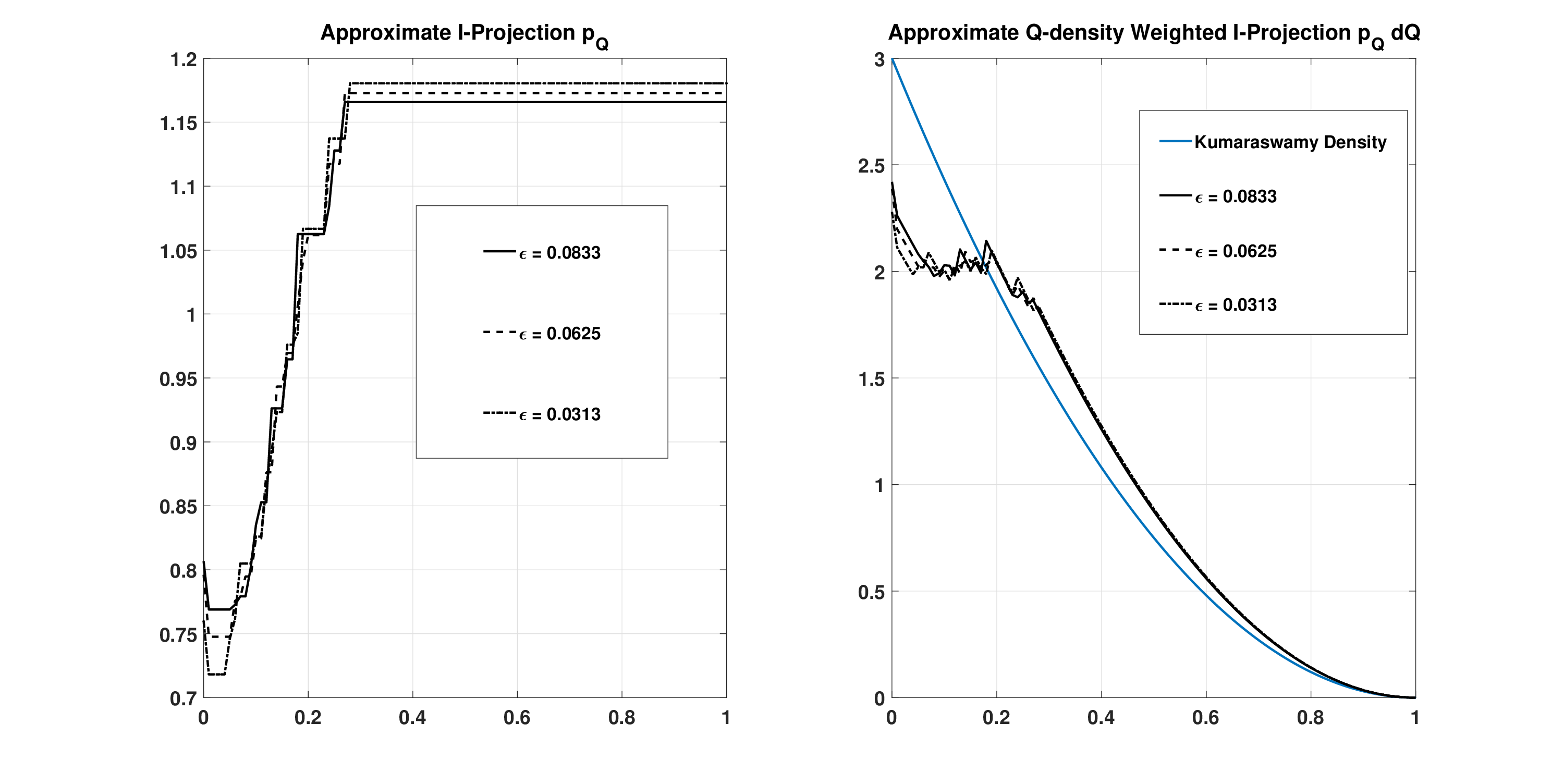}
    }
    \caption{Graphs of the approximate $I$-projection, $p_Q$, and its $Q$-density weighted version, $p_Q\,dQ$ based on the modified approximation scheme.}
    \label{fig: Random Set modified}
\end{figure}

\par The numerical results are very encouraging: we obtain results identical to those in the $a$ columns of Table~\ref{tab03} and Figure~\ref{fig: Random Set}, except that now $n$ grows less rapidly, speeding up the numerical optimization procedure with a lower number of choice variables. The $b$ columns in Table~\ref{tab03} report the results of this numerical experiment where we have also considered lower values of $\epsilon$ given by $\{0.0625,0.0313\}$ because of the improved rapid computation. Figure~\ref{fig: Random Set modified} reports the graphs of the $I$-projection and its weighted version under this modification of the proposed approximation scheme, but for $\epsilon\in\{0.0833,0.0625,0.0313\}$ to avoid visual clutter, as the graphs for higher values of $\epsilon$ are identical to those in Figure~\ref{fig: Random Set}. Focusing on the output for $\epsilon=0.0833$, observe that $n=1485$, in contrast to $n=10296$ under the approximation scheme in the first set of experiments. Furthermore, in this case, the execution time is about 1 minute, which is a substantial improvement over 27.3 hours of execution time under the approximation scheme that ignores the additional information. With this gain in computational speed, we also computed the approximate $I$-projection for $\epsilon\in\{0.0625,0.0313\}$, which also took about 8 minutes and 36 hours to complete, respectively. Of course, for even lower levels of $\epsilon$, one can expect the resulting growth in $n$ to create computational challenges for this modification of the proposed approximation scheme. It is possible to further improve the modified approximation scheme by using fewer choice variables; however, developing this approach and its generalization goes beyond the intended scope of this paper and is left for future research.

\appendix

\section{Intermediate Technical Results}\label{Appendix - Technical Lemmas}

\begin{lemma}\label{Lemma - KMT and MT}
Suppose that Assumption~\ref{Assump - L_1 and total boundedness} holds. Then the following hold:\\ $\text{ex}\left(\overline{\text{co}}(\overline{\mathcal{V}})\right)\neq\emptyset$, $\overline{\text{co}}\left(\text{ex}\left(\overline{\text{co}}(\overline{\mathcal{V}})\right)\right)=\overline{\text{co}}(\overline{\mathcal{V}})$, and
$\text{ex}\left(\overline{\text{co}}(\overline{\mathcal{V}})\right)\subset\overline{\mathcal{V}}$.
\end{lemma}
\begin{proof}
The proof proceeds by the direct method. Assumption~\ref{Assump - L_1 and total boundedness}, the set $\overline{\mathcal{V}}$ is compact in the $L_1(Q)$-norm (as it is complete and totally bounded). Now, since $L_1(Q)$ is a Banach space, it is therefore a Fr\'{e}chet space. This fact allows us to apply part (c) of Theorem~3.20 in~\citet{Rudin-Book} to the set $\overline{\mathcal{V}}$  to deduce that $\overline{\text{co}}(\overline{\mathcal{V}})$ is also compact in the same norm. Whence, we can apply the Krein-Milman Theorem (e.g., Theorem~3.23 in~\citet{Rudin-Book}) to the set $\overline{\text{co}}(\overline{\mathcal{V}})$ to deduce that its set of extreme points, $\text{ex}\left(\overline{\text{co}}(\overline{\mathcal{V}})\right)$, is nonempty, and that
$\overline{\text{co}}\left(\text{ex}\left(\overline{\text{co}}(\overline{\mathcal{V}})\right)\right)=\overline{\text{co}}(\overline{\mathcal{V}})$.
 Next, we can apply Milman's Theorem (e.g., Theorem~3.25 in~\citet{Rudin-Book}) to establish $\text{ex}\left(\overline{\text{co}}(\overline{\mathcal{V}})\right)\subset\overline{\mathcal{V}}$.
\end{proof}

\begin{lemma}\label{Lemma - Riemann Sum}
Suppose the conditions of Theorem~\ref{Thm - existence reparametrization} hold. Given $\epsilon>0$, let $U=\left\{y\in L_1(Q): \|y\|_{L_1(Q)}\leq\epsilon\right\}$. Then, there corresponds a finite partition $\{E_i\}_{i=1}^{n}$ of $\overline{\mathcal{V}}$ such that the approximation~(\ref{eq - Riemann Sum Consruction}) holds.
\end{lemma}
\begin{proof}
The proof proceeds by the direct method. The proof we present is \emph{our} solution to Exercise 23 of~\citet{Rudin-Book}, but in the context of this paper's framework.

\par First note that the neighborhood $U$ of 0 in $L_1(Q)$ is closed, balanced, and convex. The polar of $U$ is defined as $K=\left\{\Lambda\in L_1(Q)^*: |\Lambda(y)|\leq 1\,\forall y\in U\right\}$. We claim that
\begin{align}\label{eq - lemma proof Rudin 0}
U=\left\{y\in L_1(Q): |\Lambda(y)|\leq 1\,\forall \Lambda\in L_1(Q)^*\right\}.
\end{align}
It is clear that $U$ is a subset of the right side of~(\ref{eq - lemma proof Rudin 0}) which is closed. Suppose that $y\in L_1(Q)$ but $y_0\notin U$. Theorem~3.7 of~\citet{Rudin-Book} (with $U$ and $L_1(Q)$ in place of $B$ and $X$, respectively) then shows that $\Lambda(y_0)>1$ for some $\Lambda\in L_1(Q)^*$. This establishes the equality~(\ref{eq - lemma proof Rudin 0}). Consequently, to show that $z\in L_1(Q)$ is an element of $U$, we need to establish that
$|\Lambda(z)|\leq 1$ $\forall \Lambda\in L_1(Q)^*$, holds.

\par Next, construct the partitioning sets $E_i$ so that $v-v^\prime\in U$ whenever $v$ and $v^\prime$ lie the same $E_i$. Since $\overline{\mathcal{V}}$ is compact in $L_1(Q)$, there is a finite number of such sets, $n\in\mathbb{Z}_+$. Now we will show $\int_{\overline{\mathcal{V}}}v\,d\mu(v)-\sum_{i=1}^{n}\mu(E_i)v_i\in U$ $\forall v_i\in E_i$ and $i=1,\ldots,n$. Select any $v_i\in E_i$ for each $i=1,\ldots, n$, and define $z=\int_{\overline{\mathcal{V}}}v\,d\mu(v)-\sum_{i=1}^{n}\mu(E_i)v_i$. Then, for any $\Lambda\in L_1(Q)^*$,
\begin{align*}
\Lambda(z) & =\Lambda\left(\int_{\overline{\mathcal{V}}}v\,d\mu(v)\right)-\Lambda\left(\sum_{i=1}^{n}\mu(E_i)v_i\right)\quad(\text{by linearity of $\Lambda$})\\
& = \Lambda\left(\int_{\overline{\mathcal{V}}}v\,d\mu(v)\right)-\sum_{i=1}^{n}\mu(E_i)\Lambda(v_i)\quad(\text{by linearity of $\Lambda$})\\
& = \int_{\overline{\mathcal{V}}}\Lambda(v)\,d\mu(v)-\sum_{i=1}^{n}\mu(E_i)\Lambda(v_i)\quad(\text{by Definition~\ref{Def - GF integral} of the weak integral})\\
& = \sum_{i=1}^{n}\left[ \int_{E_i}\Lambda(v)\,d\mu(v)-\mu(E_i)\Lambda(v_i)\right]\quad(\text{as $\{E_i\}_{i=1}^{n}$ is a partition of $\overline{\mathcal{V}}$}).
\end{align*}
Whence,
\begin{align*}
\left|\Lambda(z)\right| & \leq \sum_{i=1}^{n}\left| \int_{E_i}\Lambda(v)\,d\mu(v)-\mu(E_i)\Lambda(v_i)\right|\\
 & = \sum_{i=1}^{n}\left| \int_{E_i}\Lambda(v)\,d\mu(v)-\int_{E_i}\,d\mu(v)\Lambda(v_i)\right|\quad(\text{as $\int_{E_i}\,d\mu(v)=\mu(E_i)$}).\\
 & = \sum_{i=1}^{n}\left| \int_{E_i}\left(\Lambda(v)-\Lambda(v_i)\right)\,d\mu(v)\right|\quad(\text{by linearity of Lebesgue integral})\\
 & = \sum_{i=1}^{n}\left| \int_{E_i}\left(\Lambda(v-v_i)\right)\,d\mu(v)\right|\quad(\text{by linearity of $\Lambda$})\\
 &\leq \sum_{i=1}^{n}\int_{E_i}\left|\Lambda(v-v_i)\right|\,d\mu(v) \\
 & \leq 1,
\end{align*}
where the last inequality follows from $v-v_i\in E_i\implies v-v_i\in U$ for each $i$, and by the right side of~(\ref{eq - lemma proof Rudin 0}). As $\Lambda\in L_1(Q)^*$ in the above derivations was arbitrary, these derivations hold for all $\Lambda\in L_1(Q)^*$. Consequently, $|\Lambda(z)|\leq 1$ for all $\Lambda\in L_1(Q)^*$, which means that $z$ is an element of the right side of~(\ref{eq - lemma proof Rudin 0}), and hence, $z\in U$. Finally, the derivations for this $z$ are for an arbitrary choice of elements $v_i\in E_i$ for each $i$. Therefore, these derivations and deductions go through for any choice of the $v_i\in E_i$. This concludes the proof.
\end{proof}

\begin{lemma}\label{Lemma - Inequality}
Suppose the conditions of Theorem~\ref{Thm - existence reparametrization} hold. Then
$\frac{\int_{\Omega}ve^{y_0}\,dQ}{\int_{\Omega}e^{y_0}\,dQ}\geq \int_{\Omega}v\,dQ\quad \forall v\in\overline{B}$.
\end{lemma}
\begin{proof}
The proof proceeds by the direct method. Let $v\in\overline{B}$. By definition of covariance,
$\int_{\Omega}ve^{y_0}\,dQ=\int_{\Omega}e^{y_0}\,dQ\,\int_{\Omega}v\,dQ+\text{COV}_{Q}\left(e^{y_0},v\right)$,
where $\text{COV}_{Q}\left(e^{y_0},v\right)$ is the covariance between $e^{y_0}$ and $v$ under $Q$. This equality is equivalent to
$\frac{\int_{\Omega}ve^{y_0}\,dQ}{\int_{\Omega}e^{y_0}\,dQ}=\int_{\Omega}v\,dQ+\frac{\text{COV}_{Q}\left(e^{y_0},v\right)}{\int_{\Omega}e^{y_0}\,dQ}$. Thus, the desired conclusion would arise if $\text{COV}_{Q}\left(e^{y_0},v\right)\geq0$, since $\int_{\Omega}e^{y_0}\,dQ>0$. From Theorem~\ref{Thm - Fenchel}, we have that $y_0\in\overline{\text{span}_{+}(B)}$, and since $\overline{B}\subset\overline{\text{span}_{+}(B)}$, we must have $\text{COV}_{Q}\left(e^{y_0},v\right)\geq0$. This is because $y_0$ must either not depend on $v$, or it would depend on $v$ linearly with a positive coefficient, which is on account of the positive linear span operation. This concludes the proof.
\end{proof}

\begin{lemma}\label{Lemma - precompactness}
Suppose that the class of moment $\mathcal{V}$ in~(\ref{eq - moment fxns set}) is precompact in the norm topology of $L_r(Q)$ for some $r>1$. Then $\mathcal{V}$ is precompact in the $L_1(Q)$ norm.
\end{lemma}
\begin{proof}
The proof proceeds by the direct method and makes use of Lemma~1 in~\citet{HANCHEOLSEN2010385}. Simply apply their result using the identity map as the choice of the function $\Phi$ in the statement of their result.
\end{proof}

\section{Technical Results for Examples~\ref{Example - SD} and~\ref{Example - Bhatacharya}}\label{Appendix - Derivations Examples}

\begin{lemma}\label{lemma - Example SD 2}
Let $\mathcal{V}$ be given as in~ Example~\ref{Example - SD}. Then this set satisfies $\mathcal{V}=\overline{\mathcal{V}}$ and Assumption~\ref{Assump - L_1 and total boundedness}.
\end{lemma}
\begin{proof}
The proof proceeds by the direct method. On establishing (i), consider an arbitrary sequence \\ $\{y_n\}_{n\geq1}\subset \mathcal{V}$ such that $y_n\stackrel{L_1(Q)}{\longrightarrow} y$. To prove the desired result, we need to establish that $y\in\mathcal{V}$. As convergence in $L_1(Q)$ implies convergence in $Q$-measure, there exists a non-random increasing sequence of integers $n_1,n_2,\ldots,$ such that $\{y_{n_k}\}_{k\geq1}$ converges to $y$ a.s.-$Q$ (e.g., see Theorem~7.2.13, in~\citet{Grimmet-Stirkazer}). That is,
\begin{align}\label{eq - Example 1 0}
y(\omega)=\lim_{k\rightarrow\infty} y_{n_k}(\omega)\quad \text{a.s.}-Q.
\end{align}
The limit~(\ref{eq - Example 1 0}) implies $\{\gamma_{n_k}\}_{k\geq0}\subset [0,\overline{\gamma}]$ holds. Thus, there exists a subsequence $\{\gamma_{n_{k_\ell}}\}_{\ell\geq1}$ such that $\lim_{\ell\rightarrow \infty}\gamma_{n_{k_\ell}}=\gamma^{\star}\in[0,\overline{\gamma}]$, by the Bolzano-Weierstrass Theorem. Combining this conclusion with the limit~(\ref{eq - Example 1 0}) yields
$y(\omega)=\lim_{\ell\rightarrow\infty} y_{n_{k_\ell}}(\omega)=G_1(\gamma^{\star})-1[\omega\leq \gamma^{\star}]$  $\text{a.s.}-Q$,
as every subsequence of $\{y_{n_k}\}_{k\geq1}$ converges to $y$ a.s.-$Q$. Therefore, $y\in\mathcal{V}$. Because the sequence $\{y_n\}_{n\geq1}\subset \mathcal{V}$ such that $y_n\stackrel{L_1(Q)}{\longrightarrow} y$, was arbitrary, we have $\mathcal{V}=\overline{\mathcal{V}}$.

\par On establishing (ii), we will show the $L_1(Q)$ bracketing number of $\mathcal{V}$ is finite. Given $\epsilon>0$, we can always find a partition of $\mathbb{R}$ of size $k$, where $0=\gamma_1<\gamma_2<\cdots<\gamma_k=\overline{\gamma}$, such that $G_1(\gamma_j\,-)-G_1(\gamma_{j-1})+\gamma_j-\gamma_{j-1}\leq \epsilon$ $\forall 1\leq j\leq k$,
with  $G_1(\gamma_j\,-)=\lim_{s\uparrow \gamma_j}G_1(s)$. Consider the collection of brackets $u_j(\omega)=G_1(\gamma_j)-1\left[\omega\leq \gamma_{j-1}\right]\quad\text{and}\quad l_j(\omega)=G_1(\gamma_{j-1}\,-)-1\left[\omega< \gamma_j\right]$ $\forall 1\leq j\leq k$. Now each element of $\mathcal{V}$ is in at least one bracket, i.e., for $\gamma_{j-1}\leq \gamma\leq \gamma_j$,
$l_j(\omega)\leq G_1(\gamma)-1[\omega\leq \gamma]\leq u_j(\omega)$ $\text{a.s.}-Q$,
and $\|u_j-l_j\|_{L_1(Q)}\leq \epsilon$. Thus, the $L_1(Q)$ bracketing number of $\mathcal{V}$ satisfies the inequalities $N_{[\,]}\left(\epsilon,\mathcal{V}, L_1(Q)\right)<\infty$ $\forall \epsilon>0$, and hence, $\mathcal{V}$ is precompact in $L_1(Q)$.
\end{proof}

\begin{lemma}\label{lemma - Example Bhattacharya}
Let $\mathcal{V}$ be given as in Example~\ref{Example - Bhatacharya}. Then this set satisfies $\mathcal{V}=\overline{\mathcal{V}}$ and Assumption~\ref{Assump - L_1 and total boundedness}.
\end{lemma}
\begin{proof}
The proof proceeds by the direct method. On establishing (i), it is sufficient to show $\mathcal{V}_i=\overline{\mathcal{V}}_i$, for each $i$. Without loss of generality, our derivations focus on the case $i=1$. Consider an arbitrary sequence $\{y_n\}_{n\geq1}\subset \mathcal{V}_1$ such that $y_n\stackrel{L_1(Q)}{\longrightarrow} y$. To prove the desired result, we need to establish that $y\in\mathcal{V}_1$. As convergence in $L_1(Q)$ implies convergence in $Q$-measure, there exists a non-random increasing sequence of integers $n_1,n_2,\ldots,$ such that $\{y_{n_k}\}_{k\geq1}$ converges to $y$ a.s.-$Q$ (e.g., see Theorem~7.2.13, in~\citet{Grimmet-Stirkazer}). That is,
\begin{align}\label{eq - Example 2 0}
y(\omega)=\lim_{k\rightarrow\infty} y_{n_k}(\omega)\quad \text{a.s.}-Q.
\end{align}
The limit~(\ref{eq - Example 2 0}) implies $\{\gamma_{n_k}\}_{k\geq0}\subset [-\infty,+\infty]$ holds. Thus, there exists a subsequence $\{\gamma_{n_{k_\ell}}\}_{\ell\geq1}$ such that $\lim_{\ell\rightarrow \infty}\gamma_{n_{k_\ell}}=\gamma^{\star}\in[-\infty,+\infty]$, because $[-\infty,+\infty]$ is the two-point compactification of $\mathbb{R}$. Combining this conclusion with the limit~(\ref{eq - Example 2 0}) yields
$y(\omega)=\lim_{\ell\rightarrow\infty} y_{n_{k_\ell}}(\omega)=\left(G_1(\gamma^{\star})-1[x_1\leq \gamma^{\star}]\right)\,1[x_2\in\mathbb{R}]$ $\text{a.s.}-Q$,
as every subsequence of $\{y_{n_k}\}_{k\geq1}$ converges to $y$ a.s.-$Q$. Therefore, $y\in\mathcal{V}_1$. Because the sequence $\{y_n\}_{n\geq1}\subset \mathcal{V}_1$ such that $y_n\stackrel{L_1(Q)}{\longrightarrow} y$, was arbitrary, we have $\mathcal{V}_1=\overline{\mathcal{V}}_1$. An identical argument holds for establishing that $\mathcal{V}_2=\overline{\mathcal{V}}_2$, which we omit for brevity.

\par On establishing (ii), it is sufficient to show $\mathcal{V}_1$ and $\mathcal{V}_2$ are precompact in $L_1(Q)$, since $\mathcal{V}$ is their union. Without loss of generality, our derivations focus on the case $i=1$, and shows that this set's $L_1(Q)$ bracketing number of $\mathcal{V}$ is finite. Given $\epsilon>0$, we can always find a partition of $\mathbb{R}$ of size $k$, where $-\infty=\gamma_1<\gamma_2<\cdots<\gamma_k=\infty$, such that
$G_1(\gamma_j\,-)-G_1(\gamma_{j-1})+Q_{X_1}(\gamma_1,\gamma_j\,-)-Q_{X_1}(\gamma_1,\gamma_{j-1})\leq \epsilon$ $\forall 1\leq j\leq k$
with  $G_1(\gamma_j\,-)=\lim_{s\uparrow \gamma_j}G_1(s)$ and  $Q_{X_1}(\gamma_1,\gamma_j\,-)=\lim_{s\uparrow \gamma_j}Q_{X_1}(\gamma_1,s)$. Consider the collection of brackets
$u_j(x_1)=G_1(\gamma_j)-1\left[1\leq \gamma_{j-1}\right]$ and $l_j(x_1)=G_1(\gamma_{j-1}\,-)-1\left[1< \gamma_j\right]$ $\forall 1\leq j\leq k$. Now each element of $\mathcal{V}_1$ is in at least one bracket, i.e., for $\gamma_{j-1}\leq \gamma\leq \gamma_j$, $l_j(x_1)\leq G_1(\gamma)-1[x_1\leq \gamma]\leq u_j(x_1)$, $\text{a.s.}-Q$, and $\|u_j-l_j\|_{L_1(Q)}\leq \epsilon$. Thus, $N_{[\,]}\left(\epsilon,\mathcal{V}_1, L_1(Q)\right)<\infty$ $\forall \epsilon>0$, and hence, $\mathcal{V}_1$ is precompact in $L_1(Q)$. We can follow identical steps to those above to deduce that $\mathcal{V}_2$ is precompact in $L_1(Q)$. We omit these details for brevity.
\end{proof}

\section{Proof of Proposition~\ref{Prop - V Random Set}}\label{Appendix - proof RS}
\begin{proof}
The proof proceeds by the direct method. First, using the Lipschitz continuity of the capacity functional $\sigma$ (i.e., inequality~(\ref{eq - V RS 0})), and following steps similar to those in the proof of Proposition~\ref{prop - unconditional SD}, but now accounting for the different forms of the moment functions $f_\gamma$ and $\omega$, we can establish that $\mathcal{V}=\overline{\mathcal{V}}$ holds. Specifically, the key adjustment entails using the compactness of $\mathcal{C}(\Omega)$ (with the Hausdorff metric) as it implies that it must have the Bolzano-Weierstrass Property, which we exploit in the proof of Proposition~\ref{prop - unconditional SD} so that we can calculate the $L_1(Q)$ limit of sequences in $\mathcal{V}$ and deduce that they must be elements of $\mathcal{V}$. We omit the details for brevity. This result implies that $\mathcal{V}$ is closed in norm topology of $L_1(Q)$.

\par Next, we establish that $\mathcal{V}$ is precompact. Denote by $\ominus$ the symmetric difference set operation on $C(\Omega)$. Given $\epsilon>0$, denote by $N\left(\epsilon,\mathcal{V}, L_1(Q)\right)$ the minimum number of $\epsilon$-balls in the $L_1(Q)$-norm needed to ensure that every $y\in\mathcal{V}$ lies in at least one such ball. The precompactness of $\mathcal{V}$ can be established by showing that $N\left(\epsilon,\mathcal{V}, L_1(Q)\right)<+\infty$ $\forall \epsilon>0$. Towards that end, for any $\rho>0$, $\exists n\in\mathbb{Z}_+$ such that $A_i\in \mathcal{C}(\Omega)$ for $i=1,\ldots,n$ with $\mathcal{C}(\Omega)\subseteq\bigcup_{i=1}^{n}B(A_i,\rho)$, where $B(A_i,\rho)$ is a $\rho$-ball in $\mathcal{C}(\Omega)$. Next, consider the sets $\left\{-f_\gamma:\gamma\in B(A_i,\rho)\right\}$ and let $-f_{\gamma_i}$ be the moment function with $\gamma=A_i$, for $i=1,\ldots,n$. Fix $i$, and let $f_\gamma\in\left\{-f_\gamma:\gamma\in B(A_i,\rho)\right\}$. Then we have
\begin{align*}
\int_{\Omega}|f_{\gamma}-f_{\gamma_i}|\,dQ & \leq \left|\sigma(A)-\sigma(A_i)\right|+Q\left(A\ominus A_i\right)\;(\text{by the form of $f_{\gamma}$})\\
 & \leq 2 d_{H}\left(A,A_i\right)+Q\left(A\ominus A_i\right)\;(\text{by~(\ref{eq - V RS 0})})\\
 & \leq 2 \rho+Q\left(A\ominus A_i\right)\;(\text{by precompactness of $\mathcal{C}(\Omega)$}) \\
 & \leq 2 \rho +2\rho=4\rho.
\end{align*}
Thus, the set $\left\{-f_\gamma:\gamma\in B(A_i,\rho)\right\}$ is a ball in $\mathcal{V}$ based on the $L_1(Q)$-norm which is centered at $-f_{\gamma_i}$ having radius $8\rho$. As $\mathcal{V}\subseteq \bigcup_{i=1}^{n}\left\{-f_\gamma:\gamma\in B(A_i,\rho)\right\},$ we can set $\rho=\epsilon/4$, to obtain $n$ balls in $\mathcal{V}$ using the desired norm, each with radius $\epsilon$, having centers $\left\{-f_{\gamma_i},i=1,\ldots,n\right\}$, which cover $\mathcal{V}$. This shows $N\left(\epsilon,\mathcal{V}, L_1(Q)\right)<\infty$ for an arbitrary $\epsilon>0$, and hence, the above argument holds for all $\epsilon>0$.
\end{proof}


\bibliographystyle{chicago}
\bibliography{mcgilletd}
\end{document}